\documentclass[11 pt,leqno]{article}
\usepackage{amsmath}
\usepackage{amsfonts}
\usepackage{amssymb}
\usepackage{amsthm}
\usepackage{latexsym}
\usepackage{epsfig}
\usepackage{xypic}
\usepackage{verbatim}
\usepackage{mathrsfs}
\usepackage{pinlabel}
\usepackage{url}

\theoremstyle{plain} \newtheorem{theorem}{Theorem}[subsection]
\theoremstyle{plain} \newtheorem{proposition}[theorem]{Proposition}
\theoremstyle{plain} \newtheorem{lemma}[theorem]{Lemma}
\theoremstyle{plain} 
\theoremstyle{plain} \newtheorem{corollary}[theorem]{Corollary}
\theoremstyle{plain} \newtheorem{claim}[theorem]{Claim}
\theoremstyle{plain} 
\theoremstyle{plain} 
\theoremstyle{plain} 
\theoremstyle{plain} 
\theoremstyle{remark} 
\theoremstyle{definition} \newtheorem{remark}[theorem]{Remark}
\theoremstyle{definition} \newtheorem{definition}[theorem]{Definition}
\theoremstyle{definition} \newtheorem{theorem-definition}[theorem]{Theorem-Definition}
\theoremstyle{definition} 
\theoremstyle{plain} \newtheorem{maintheorem}{Theorem}
\theoremstyle{plain} 
\theoremstyle{plain} \newtheorem{maincorollary}[maintheorem]{Corollary}

\numberwithin{equation}{section}

\newcommand\Hom{\mathrm{Hom}}
\newcommand\Diff{\mathrm{Diff}}

\newcommand\R{\mbox{$\mathbb{R}$}}

\newcommand\Z{\mbox{$\mathbb{Z}$}}

\newcommand\into\hookrightarrow

\newcommand\bs{\setminus}

\def\co{\colon\thinspace}
\newcommand\ptI{\mathcal{I}_*}
\newcommand\ptIg{\mathcal{I}_{g,*}}
\newcommand\ptM{\mathrm{Mod}(\Sigma,*)}

\newcommand\clM{\mathrm{Mod}(\Sigma)}
\newcommand\clMn{\mathrm{Mod}}

\newcommand\johom{\tau}

\newcommand\Fl{\mathrm{Flux}}
\newcommand\Symp{\mathrm{Symp}(\Sigma)}
\newcommand\Sympn{\mathrm{Symp}}
\newcommand\Symps{\mathrm{Symp}(\Sigma,*)}
\newcommand\ISymp{\mathrm{ISymp}(\Sigma)}
\newcommand\ISymps{\mathrm{ISymp}(\Sigma,*)}

\newcommand\Sympo{\mathrm{Symp}_0(\Sigma)}
\newcommand\Sympon{\mathrm{Symp}_0}
\newcommand\Sympso{\mathrm{Symp}_0(\Sigma,*)}
\newcommand\Ham{\mathrm{Ham}(\Sigma)}
\newcommand\Hamn{\mathrm{Ham}}
\newcommand\Flsec[1]{A_{#1}}
\newcommand\Flmet[1]{\widetilde F_{#1}}
\newcommand\Fljac[1]{\Fl^X_{#1}}
\newcommand\angb[1]{\langle #1\rangle}
\newcommand\ai{\hat \imath} 
\newcommand\supp{\mathrm{supp}\,}
\title{Extended flux maps on surfaces and the contracted Johnson homomorphism}
\author{Matthew B. Day}
\date{September 3, 2009}
\begin{document}
\maketitle
\begin{abstract}
On a closed symplectic surface $\Sigma$ of genus two or more, we give a new construction of an extended flux map (a crossed homomorphism from the symplectomorphism group $\mathrm{Symp}(\Sigma)$ to the cohomology group $H^1(\Sigma;\mathbb{R})$ that extends the flux homomorphism).
This construction uses the topology of the Jacobian of the surface and a correction factor related to the Johnson homomorphism.
For surfaces of genus three or more, we give another new construction of an extended flux map using hyperbolic geometry.
\end{abstract}

\section{Introduction}
\subsection{Background}
Let $\Sigma=\Sigma_{g,*}$ be a closed, oriented surface of genus $g\geq 2$ with a basepoint~$*$.
Let $\omega_\Sigma$ be a symplectic form (an area form) on $\Sigma$.
Let $\Symp=\Sympn(\Sigma;\omega_\Sigma)$ be the symplectomorphism group of $(\Sigma,\omega_\Sigma)$, which is the subgroup of diffeomorphisms $\phi\in\Diff(\Sigma)$ with $\phi^*\omega_\Sigma=\omega_\Sigma$.
Let $\Sympo=\Sympon(\Sigma;\omega_\Sigma)$ be the subgroup of $\Symp$ consisting of symplectomorphisms homotopic to the identity.

There is a homomorphism $\Fl\co \Sympo\to H^1(\Sigma;\R)$ called the \emph{flux homomorphism}.
By the universal coefficient theorem, we represent $\Fl$ as a homomorphism to $\Hom(H_1(\Sigma),\R)$.
Let $\phi\in \Sympo$, and
pick a smooth homotopy $\phi_t$ from the identity to $\phi=\phi_1$.
Represent a class in $H_1(\Sigma)$ by
$\gamma_*[S^1]$ for a smooth $\gamma\co S^1\to\Sigma$.
Define $K\co S^1\times [0,1]\to\Sigma$ by dragging $\gamma$ along $\phi_t$; specifically, $K(x,t)=\phi_t(\gamma(x))$.
Then
\[\Fl(\phi)(\gamma_*[S^1])=\int_{K_*[S^1\times[0,1]]}\omega_\Sigma.\]
Intuitively, $\Fl(\phi)(\gamma_*[S^1])$ measures the area that moves across $\gamma_*[S^1]$ in flowing from the identity to $\phi$. 
See Section~10.2 of McDuff--Salamon~\cite{mcdsal} 
for discussion of flux homomorphisms.  
Also see Remark~\ref{re:surfflux} below.

A \emph{crossed homomorphism} from a group $G$ to a $G$--module $M$ is a map $f\co G\to M$ that obeys a twisted homomorphism identity:
\[f(a b)=a\cdot f(b)+f(a),\quad \text{ for $a,b\in G$}.\]
This is the same thing as a group cohomology $1$--cocycle in $Z^1(G;M)$.
The group $H^1(\Sigma;\R)$ is a left $\Symp$--module via the action $(\phi\cdot\alpha)([c])=\alpha(\phi_*^{-1}[c])$, for $\phi\in\Symp$, $\alpha\in H^1(\Sigma;\R)$ and $[c]\in H_1(\Sigma;\R)$.
An \emph{extended flux map} is a crossed homomorphism $\Symp\to H^1(\Sigma;\R)$ that extends the flux homomorphism.
The study of extended flux maps was initiated by Kotschick--Morita~\cite{km}, who showed the existence of a cohomology class of maps extending flux.
McDuff gave a concrete example of an extended flux map on surfaces in Remark~4.7 of~\cite{mcd}.

Let $\Symps=\Sympn(\Sigma,*;\omega_\Sigma)$ be the subgroup of symplectomorphisms in $\Symp$ fixing the basepoint $*\in\Sigma$, and let $\Sympso$ be the connected component of the identity.
In this paper, we construct two crossed homomorphisms from $\Symps$ to $H_1(\Sigma;\R)$ that agree with $\Fl$ on $\Sympso$.
The first is a modified version of McDuff's extended flux map from~\cite{mcd}, and the second uses the Jacobian torus of $\Sigma$ to measure flux.
We compute the differences of these maps with the restriction of an extended flux map.
Our computation shows that these maps do not agree with $\Fl$ on $\Symps\cap\Sympo$, which is strictly larger than $\Sympso$.
We then explain a way to modify these constructions to produce extended flux maps.
We end the paper by giving a new construction of an extended flux map on a surface of genus $g\geq 3$, using hyperbolic geometry.

\subsection{Summary of basepoint-preserving constructions}
First we define a crossed homomorphism on $\Symps$ that is a slight variation of McDuff's definition of an extended flux map on surfaces from~\cite{mcd}, Remark~4.7.
One interpretation of $\Fl$ is that it measures the area cobounded by a cycle and its push-forward under a symplectomorphism in $\Sympo$ (so that the push-forward cycle is homologous).
This idea can be extended to arbitrary symplectomorphisms by fixing a set of specific cycles for reference.
Fix a homomorphism $s\co H_1(\Sigma)\to Z_1(\Sigma\bs\{*\})$ such that the piecewise-smooth cycle $s([c])$ represents $[c]$ for each $[c]\in H_1(\Sigma)$.
We define the \emph{section-based area difference map} based on $s$,
\[\Flsec{s}\co \Symp\to H^1(\Sigma;\R)\]
by setting $\Flsec{s}(\phi)([c])$ to be the area cobounded by $s(\phi^{-1}_*[c])$ and $\phi^{-1}_*s([c])$.
We use the basepoint in this construction to decide which of the regions bounded by these two cycles to measure. 
We give our definition more precisely and prove basic lemmas about this map in Section~\ref{ss:flsec} below.

Our second construction is a crossed homomorphism on $\Symps$ that measures flux in the Jacobian torus of the surface.
Let 
\[X=H_1(\Sigma;\R)/H_1(\Sigma;\Z)\cong (S^1)^{2g}\] 
be the Jacobian torus of $\Sigma$.
As we explain in Section~\ref{ss:jacobianprelim}, $X$ has a natural symplectic structure and there is a symplectic Abel--Jacobi map $J\co \Sigma\to X$.
We also show in that section that this forces a particular normalization convention, namely that $\mathrm{Area}(\Sigma,\omega_\Sigma)=g$.

The action of $\Symp$ on $H_1(\Sigma;\R)$ induces a basepoint-preserving, symplectic action on $X$.
Denote this action by $\rho\co \Symp\to\Diff(X)$.
Note that for any $\phi\in\Symps$, we have
\[(J\circ \phi^{-1})_*=(\rho(\phi^{-1})\circ J)_*\co\pi_1(\Sigma,*)\to\pi_1(X,0).\]
Since these smooth maps are between aspherical manifolds and they induce the same map on fundamental groups, we have a smooth homotopy $K\co \Sigma\times[0,1]\to X$ from  $J\circ \phi^{-1}$ to $\rho(\phi^{-1})\circ J$, relative to the basepoint.
Given a cycle $c\in Z_1(\Sigma)$, we can use $K$ to measure a kind of flux of $\phi$ across $c$.
Specifically, define a chain $C\in C_2(X)$ by dragging $J_*c$ along $K$.
The \emph{Jacobian flux crossed homomorphism} 
\[\Fljac{J}\co \Symps\to H^1(\Sigma;\R)\]
is given by $\Fljac{J}(\phi)([c])=\int_C\omega_X$, with $\phi$, $c$ and $C$ as above.
We state this definition more precisely in Section~\ref{ss:fljac}.
We also show in that section that $\Fljac{J}$ is a well-defined crossed homomorphism agreeing with $\Fl$ on ${\Sympso}$ and that it is independent of all the choices except~$J$.

\subsection{Comparison results}

Next we compare $\Flsec{s}$ and $\Fljac{J}$ to an extended flux map by computing their differences.
Since these maps agree on $\Sympso$, their differences are constant on each connected component of $\Symps$ and are topological invariants.
In Theorem~\ref{mt:hypvssec} and Theorem~\ref{mt:jacvssec} below, we identify these topological invariants.

Let $\ptM=\Diff^+(\Sigma,*)/\Diff^0(\Sigma,*)$ be the mapping class group of $\Sigma$ relative to $*$, which is the group of orientation-preserving diffeomorphisms of $\Sigma$ fixing $*$ modulo equivalence by homotopy relative to $*$.
The \emph{Torelli group} $\ptI=\ptIg$ is the subgroup of $\ptM$ of classes acting trivially on $H_1(\Sigma)$.
There is a homomorphism $\johom\co \ptI\to\bigwedge^3H_1(\Sigma)$ called the \emph{Johnson homomorphism}.
We discuss this map in Section~\ref{ss:johnsonprelim}.

Let $\ISymp$ denote the subgroup of $\Symp$ acting trivially on $H_1(\Sigma)$, and let $\ISymps$ denote its subgroup fixing $*$.
As we explain in Section~\ref{ss:diffmetsec}, every extended flux map has the same restriction to $\ISymp$.
The projection $\Diff^+(\Sigma,*)\to\ptM$ restricts to a map $p\co \ISymps\to\ptI$.
Let $D_\Sigma\co H^1(\Sigma)\to H_1(\Sigma)$ denote the Poincar\'e duality map.
In Section~\ref{ss:johnsonprelim}, we discuss the symplectic contraction map $\Phi\co \bigwedge^3H_1(\Sigma)\to H_1(\Sigma)$.
The constants of proportionality in the following theorems depend on our convention that
 $\mathrm{Area}(\Sigma,\omega_\Sigma)=g$.

\begin{maintheorem}\label{mt:hypvssec}
For any choice of $s$ as above and any extended flux map $F$, 
\[(\Flsec{s}-F)|_{\ISymps}=\frac{g}{g-1}D_\Sigma^{-1}\circ\Phi\circ\johom\circ p.\]
\end{maintheorem}

\begin{maintheorem}\label{mt:jacvssec}
For any choices of $s$ and $J$ as above,
\[(\Flsec{s}-\Fljac{J})|_{\ISymps}=D_\Sigma^{-1}\circ\Phi\circ\johom\circ p.\]
\end{maintheorem}

It follows immediately from these theorems that neither $\Flsec{s}$ nor $\Fljac{J}$ agrees with $\Fl$ on $\Symps\cap\Sympo$ (for any choice of $s$ or $J$).
However, we can use our methods to recover true extended flux maps from $\Flsec{s}$ and $\Fljac{J}$.
Of course in the case of $\Flsec{s}$,  McDuff has constructed an extended flux map this way, but this is an alternative approach.

\begin{maincorollary}\label{mc:eflmaps}
Let $\epsilon\co \ptM\to H^1(\Sigma;\R)$ be any crossed homomorphism extending $\Phi\circ\johom$, and let $s$ and $J$ be any choices as above.
Then 
\[\Flsec{s}-\frac{g}{g-1}D_\Sigma\circ\epsilon\circ p\co \Symps\to H^1(\Sigma;\R)\]
and
\[\Fljac{J}+\frac{1}{g-1}D_\Sigma\circ\epsilon\circ p\co\Symps\to H^1(\Sigma;\R)\]
extend to extended flux maps on $\Symp$.
\end{maincorollary}
Proposition~\ref{pr:ext} below shows that such an $\epsilon$ exists.

\subsection{An extended flux map via hyperbolic geometry}

We conclude our discussion by using a hyperbolic metric on $\Sigma$ to construct an extended flux map.
This construction works only if the genus $g$ of the surface is greater than or equal to three.
We set some more notation before describing it.
The mapping class group $\clM=\Diff^+(\Sigma)/\Diff^0(\Sigma)$ of $\Sigma$ (not preserving a basepoint) is the group of orientation-preserving diffeomorphisms of $\Sigma$ modulo equivalence by free homotopy.
The group $\Ham=\Hamn(\Sigma;\omega_\Sigma)$ of \emph{Hamiltonian symplectomorphisms} is the kernel of $\Fl$ on $\Sympo$.
The map $\Fl$ induces an isomorphism $\Sympo/\Ham\cong H^1(\Sigma;\R)$, so $H^1(\Sigma;\R)$ embeds in $\Symp/\Ham$.
Therefore there is an exact sequence:
\[1\to H^1(\Sigma;\R)\to \Symp/\Ham\to \clM \to 1.\]
The sequence splits because an extended flux map is the first-coordinate map in an isomorphism $\Symp/\Ham\cong H^1(\Sigma;\R)\rtimes \clM$.

Let $h$ be a hyperbolic metric on $\Sigma$ such that the hyperbolic area form $dV_h$ is a constant multiple of $\omega_\Sigma$
(we construct such an $h$ in Section~\ref{ss:hyperbolicprelim}).
Below, we use $h$ to construct a splitting $\hat\sigma_h\co\clM\to\Symp/\Ham$.
We show that the following map is well defined in Section~\ref{ss:hypmap}.
\begin{definition}
The \emph{hyperbolic metric extended flux map} with respect to~$h$ 
is the crossed homomorphism 
\[\Flmet{h}\co \Symp\to H^1(\Sigma;\R)\]
 induced by the splitting $\hat\sigma_h$, that is, $\Flmet{h}(\phi)=\Fl(\phi\hat\sigma_h([\phi])^{-1})$.
\end{definition}

To define this $\hat\sigma_h$, we use Dehn twists.
For a simple closed curve $a$ in $\Sigma$, the \emph{Dehn twist} $T_a\in\clM$ is the element of $\clM$ described by cutting $\Sigma$ along $a$ and regluing by a full twist along this curve to the left.
For our construction, we define a class of representatives of Dehn twists called \emph{symmetric symplectic Dehn twists}.
These are symplectomorphisms representing Dehn twists that are symmetrically centered around simple closed curves.
The definition is Section~\ref{ss:hypmap}.

\begin{maintheorem}\label{mt:symmsect}
Suppose the genus $g$ of $\Sigma$ is greater than or equal to three.
Let $\sigma\co\clM\to\Symp$ be any set-map such that for each $\phi\in\clM$, $\sigma(\phi)$ is a composition of symmetric symplectic Dehn twists around simple closed $h$--geodesics.
Then the map 
\[\hat \sigma_h\co\clM\to\Symp/\Ham,\]
given by $\hat\sigma_h(\phi)=\sigma(\phi)\cdot\Ham$, is
an injective homomorphism that is a section to the natural projection.
Further, $\hat \sigma_h$ depends only on $h$ and is independent of the choice of $\sigma$.
\end{maintheorem}
It is always possible to construct such a $\sigma$ since $\clM$ is generated by Dehn twists (see Section~\ref{ss:Dehn}).
The proof of Theorem~\ref{mt:symmsect} appears in Section~\ref{ss:hypmap}.
The main lemma in this proof is that if a product of symmetric symplectic Dehn twists around $h$--geodesics is in $\Sympo$, then it is in $\Ham$.
We use a presentation of $\clM$ due to Gervais~\cite{ge} to characterize such products. 
The only properties of the hyperbolic metric that we use are the minimal intersection of geodesics and the Gauss--Bonnet theorem.
The relevant facts about mapping class groups are discussed in Section~\ref{ss:mcgprelim} and the preliminaries on hyperbolic metrics appear in Section~\ref{ss:hyperbolicprelim}.
The proof of Theorem~\ref{mt:symmsect} uses the map $\Flsec{s}$ but is otherwise independent of the other results in the paper.

\subsection{Layout of the paper}
Section~\ref{se:prelim} contains preliminary results, conventions, and several results that are quoted from the literature.
The definition of $\Flsec{s}$ appears in Section~\ref{ss:flsec} and the definition of $\Fljac{J}$ is in Section~\ref{ss:fljac}.
We prove Theorem~\ref{mt:hypvssec} in Section~\ref{ss:diffmetsec} and Theorem~\ref{mt:jacvssec} in Section~\ref{ss:pfjacvssec}.
The proof of Corollary~\ref{mc:eflmaps} appears in Section~\ref{ss:corollary}.
The proof of Theorem~\ref{mt:symmsect} and discussion of $\Flmet{h}$ are in Section~\ref{se:hypmap}.

\subsection{Acknowledgements}
Thanks to Benson Farb for many helpful conversations about this project and for comments on a draft of this paper.
Thanks to Paul Seidel and Mohammed Abouzaid for conversations.
Thanks to Justin Malestein and Joel Louwsma for many useful comments on earlier versions of this paper.
I would also like to thank an anonymous referee 
whose suggestions have hopefully improved this paper's readability.
I gratefully acknowledge the support of the National Science Foundation.
The work presented in this paper was done under the support of an N.S.F. graduate research fellowship, and this paper was prepared under the support of an N.S.F. postdoctoral research fellowship.

\section{Preliminaries and Conventions}\label{se:prelim}
All spaces are assumed to have smooth structures and all maps are assumed to be piecewise smooth.
Homology of spaces is singular homology, and chains are assumed to be piecewise smooth.
Cohomology of spaces is de\,Rham cohomology, but we often express cohomology classes in $H^1(\Sigma;\R)$ as elements of $\Hom(H_1(\Sigma),\R)$.
\subsection{Surfaces}\label{ss:surf}
\subsubsection{Moser stability for surfaces}
We use the following result.
\begin{theorem}[Moser stability theorem for surfaces]\label{th:moser}
Given two symplectic forms $\omega_1,\omega_2$ on $\Sigma$, there is an isotopy $\phi_t\in\Diff(\Sigma)$ with $\phi_0$ the identity and $\phi_1^*\omega_1=\omega_2$, if and only if $\omega_1$ and $\omega_2$ determine the same cohomology class, meaning $\mathrm{Area}(\Sigma,\omega_1)=\mathrm{Area}(\Sigma,\omega_2)$.
Further, if $\omega_1(q)=\omega_2(q)$ for all points $q$ on a closed submanifold $Q$ of $\Sigma$, we may assume that $\phi_t$ is the identity on $Q$.
\end{theorem}
This is essentially Exercise~3.21(i) of McDuff--Salamon~\cite{mcdsal}, and it follows easily from the general Moser stability theorem (Theorem~3.17 of~\cite{mcdsal}).

\begin{remark}\label{re:surfflux}
We refer the reader to Section~10.2 of McDuff--Salamon~\cite{mcdsal} for discussion of Flux homomorphisms.
As they explain, $\Fl$ is a well-defined homomorphism from the universal cover of $\Sympo$ to $H^1(\Sigma;\R)$.
For completeness, we mention that $\Sympo$ is simply connected.
This follows from Theorem~\ref{th:moser} and from the theorem of Earle--Eells~\cite{ee} that the connected component of the identity in $\Diff(\Sigma)$ is contractible.
\end{remark}

\subsubsection{Intersection numbers}
We use the \emph{algebraic intersection form}, which is a bilinear alternating function $\ai\co \bigwedge^2 H_1(\Sigma)\to \Z$.
As usual, to compute the algebraic intersection number of two classes in $H_1(\Sigma)$, find representative  cycles that meet transversely and sum the signs of their finitely many intersection points.
We refer to the extension of $\ai$ to $\bigwedge^2 H_1(\Sigma;\R)$ by the same symbol.

We also use geometric intersection numbers.
For $a$ and $b$ closed curves or homotopy classes of closed curves, we denote the \emph{geometric intersection number} of $a$ and $b$ by $|a\cap b|$.
This is the minimum number intersection points between any pair of curves representing the homotopy classes of the curves. 
\subsubsection{A fixed set of homology basis representatives}\label{sss:basisreps}
At this point we fix a set of representatives for a homology basis in a specific configuration.

Let $x_1,\ldots,x_g,y_1,\ldots,y_g$ be a set of simple closed curves such that:
\begin{itemize}
\item  $[x_1],\ldots,[x_g],[y_1],\ldots,[y_g]$ is a basis for $H_1(\Sigma;\Z)$,
\item  $\ai([x_i],[y_i])=1$ for each $i$ and all other algebraic intersection numbers among these basis elements are zero,
\item $|x_i \cap y_i|=1$ for each $i$ and all other geometric intersection numbers among these basis representatives are zero, and
\item the basepoint $*$ does not lie on any $x_i$ or $y_i$.
\end{itemize} 

Let $x'_g$ be a simple closed curve in $\Sigma\bs\{ *\}$ that is homotopic to $x_g$ in $\Sigma$ but not in $\Sigma\bs\{ *\}$, such that $|x'_g\cap y_g|=1$ and such that $x'_g$ does not intersect any other basis representative. 
See Figure~\ref{fi:basisdiagram} for reference.

\begin{figure}[ht!]
\labellist
\small\hair 2pt
\pinlabel $x_1$ [l] at 68 107
\pinlabel $y_1$ [t] at 57 49
\pinlabel $x_{2}$ [l] at 162 111
\pinlabel $y_{2}$ [t] at 149 49
\pinlabel $\cdots$ at 208 72
\pinlabel $x_g$ [l] at 282 106
\pinlabel $x_g'$ [l] at 282 31
\pinlabel $y_g$ [t] at 305 55

\pinlabel $*$ [t] at 344 69
\endlabellist
\centering
\includegraphics[scale=0.70]{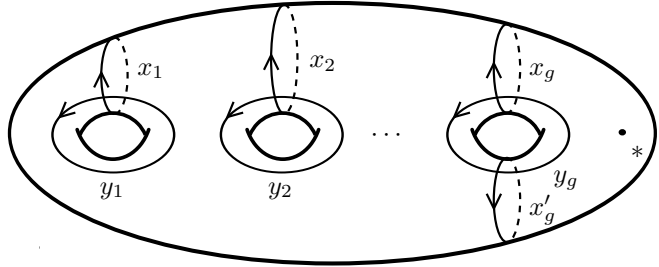}
\caption{A basis and an additional curve.}
\label{fi:basisdiagram}
\end{figure}

\subsection{Mapping class groups}\label{ss:mcgprelim}
In our arguments we consider both basepoint-preserving and non-basepoint-preserving mapping class groups of closed surfaces.
The book by Farb and Margalit~\cite{fm} is an excellent reference.
As mentioned in the introduction, the mapping class group $\clM$ of $\Sigma$ is the group of orientation-preserving homeomorphisms of $\Sigma$, modulo equivalence by free homotopy.
The basepoint-preserving mapping class group $\ptM$ of $\Sigma$ is the group of orientation-preserving homeomorphisms of $\Sigma$ preserving the basepoint, modulo equivalence by homotopy relative to the basepoint.

\subsubsection{Dehn twists}\label{ss:Dehn}
For a simple closed curve $a$ in $\Sigma$, we can find a neighborhood $A$ of $a$ that is a compact annulus.
Visualize $\Sigma$ as the boundary of a handlebody in $\R^3$, such that $A$ embeds as a long cylinder.
We can describe a homeomorphism of $\Sigma$ by cutting along $a$, rotating $A$ a full twist to the left on one side, and regluing.
The mapping class this defines in $\clM$ is the \emph{Dehn twist} around $a$ and is denoted $T_a$. 
This is illustrated in Figure~\ref{fi:Dehn}.
As long as the basepoint $*$ does not lie on $a$, we can twist on a neighborhood of $a$ not containing $*$ to get Dehn twists in $\ptM$ as well.
A more formal definition appears in Farb--Margalit~\cite{fm}.

\begin{figure}
\labellist
\small\hair 1pt
\pinlabel $a$ [b] at 58 42
\pinlabel $a$ [b] at 207 42
\pinlabel $\to$ at 135 63
\endlabellist
\centering
\includegraphics[scale=0.65]{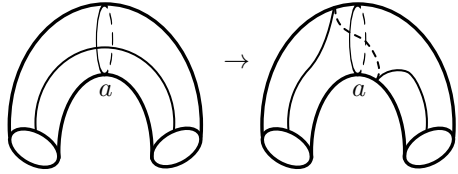}

\caption{A Dehn twist around a simple closed curve $a$, and a reference arc crossing $a$.}\label{fi:Dehn}
\end{figure}

We use the convention that positive Dehn twists turn to the left.
Since we have a fixed orientation for $\Sigma$, turning ``to the left" is well defined.
Two Dehn twists around homotopic curves represent the same element of $\clM$, and Dehn twists around curves that are homotopic in $\Sigma\backslash *$ represent the same element of $\ptM$.
Therefore we sometimes refer to the Dehn twist around a curve that is only defined up to homotopy.

We use the following theorem of Dehn.
Farb--Margalit~\cite{fm} is a reference.
\begin{theorem}[Dehn]\label{th:Dehn}
The groups $\clM$ and $\ptM$ are generated by finitely many Dehn twists around non-separating simple closed curves.
\end{theorem}

The following is a standard fact.
For $a$ a simple closed curve and $T_a$ the twist around $a$ (in $\clM$ or in $\ptM$), we have 
\begin{equation}\label{eq:twistact}(T_a)_*[b]=[b]+\ai([b],[a])[a]\end{equation}
for any $[b]\in H_1(\Sigma)$.
To verify this, construct a set of basis representatives for $H_1(\Sigma)$ in which one of the representatives intersects $a$ once, and the other representatives are disjoint from $a$.
Then it is enough verify the equation when $b$ intersects $a$ once, which is easy.

\subsubsection{Presenting the mapping class group}
We use a presentation of the mapping class group due to Gervais~\cite{ge}.
Before stating the presentation, we explain the relations it uses.

The \emph{braid relation} states that if $a$ and $b$ are simple closed curves with $|a\cap b|=1$ and $c$ is in the homotopy class of $T_a(b)$, then 
\[T_c=T_aT_bT_a^{-1}.\]

The \emph{star relation} involves simple closed curves $a_1, a_2, a_3$ and $b$ such that the $a_i$ are disjoint and each $a_i$ intersects $b$ only once and positively.
A regular neighborhood of the union of these curves is a subsurface of genus one with three boundary components.
Let $d_1, d_2$ and $d_3$ be the boundary curves of this subsurface.
This is illustrated in Figure~\ref{fi:starrelation}.
Then the star relation states:
\[(T_{a_1}T_{a_2}T_{a_3}T_b)^3=T_{d_1}T_{d_2}T_{d_3}.\]
This is called a non-separating star relation if each $d_i$ is a non-separating curve.

The \emph{chain relation} is the star relation when one of $d_1,d_2,$ or $d_3$ bounds a disk; it is called a non-separating chain relation if the other two are non-separating.

\begin{figure}[ht!]
\labellist
\small\hair 1pt
\pinlabel $a_1$ [b] at  138 124
\pinlabel $a_2$ [bl] at  68 118
\pinlabel $a_3$ [r] at   88  50
\pinlabel $b$ [b] at    100 115
\pinlabel $d_1$ [tr] at  44  56
\pinlabel $d_2$ [tl] at 163  64
\pinlabel $d_3$ [b] at  101 157
\endlabellist
\centering
\includegraphics[scale=0.65]{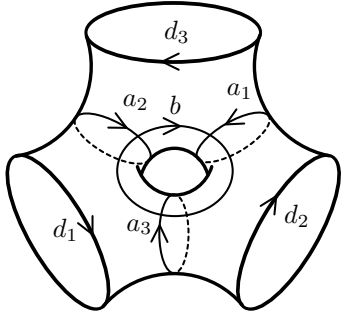}
\caption{The curves of a star relation.} 
\label{fi:starrelationnoref}
\end{figure}

The following is Theorem~B from Gervais~\cite{ge}.
\begin{theorem}[Gervais's Theorem]\label{th:presentation}
If $g\geq 3$, the group $\clM$ has a presentation with the set of Dehn twists around all free homotopy classes of essential simple closed curves as its generators and with the following relations:
\begin{enumerate}
\item\label{it:comm}  that $T_a$ and $T_b$ commute for all pairs of disjoint simple closed curves $a$ and $b$,
\item\label{it:braid} the braid relations $T_c=T_aT_bT_a^{-1}$ for all pairs of simple closed curves $a$ and $b$ that intersect once, where $c$ is in the free homotopy class of $T_a(b)$, 
\item\label{it:chain} a single non-separating chain relation, and
\item\label{it:star} a single non-separating star relation.
\end{enumerate}
\end{theorem}
This presentation has infinitely many generators and infinitely many relations, even though $\ptM$ is a finitely presentable group.
However, it is appealing because it is essentially basis independent.
Gervais also has a presentation for the mapping class group of a surface of genus two, but we will not use it.

\begin{remark}
If $a$ and $b$ are simple closed curves that intersect once, then
Gervais's braid relation easily implies the relation $T_aT_bT_a=T_bT_aT_b$.
This is what most authors mean by a ``braid relation''.
\end{remark}

\subsubsection{The Birman exact sequence}
For $\phi$ a diffeomorphism fixing $*$, $\phi$ represents classes both in $\ptM$ and in $\clM$.
This defines a natural map $\ptM\to\clM$.

For an embedding $\gamma\co S^1\to \Sigma$ based at $*$, we can find an isotopy $\phi_t\co \Sigma\times[0,1]\to\Sigma$ such that $\phi_t(*)=\gamma(t)$.
The time-one map $\phi_1$ is a \emph{point-pushing map} of $\gamma$; it clearly represents the trivial element of $\clM$, but represents a well-defined, nontrivial element of $\ptM$ if $\gamma$ is not homotopically trivial.

The following theorem of Birman 
shows that point-pushing maps generate the kernel of $\ptM\to\clM$.
Farb--Margalit~\cite{fm} is a reference.
\begin{theorem}[The Birman exact sequence]
\label{th:birmanexact}
There is an exact sequence
\[1\to \pi_1(\Sigma,*)\to \ptM \to \clM\to 1.\]
The inclusion map is given on the classes of simple closed loops (which generate $\pi_1(\Sigma,*)$) by taking point-pushing maps, and the projection map is the natural map.
\end{theorem}
A point-pushing map can be easily expressed as a product of Dehn twists.
In particular, $T_{x_g}T_{x'_g}^{-1}$ is a point-pushing map, along a curve parallel to $x_g$.
By Equation~(\ref{eq:twistact}),  $T_{x_g}T_{x'_g}^{-1}$ is in $\ptI$.
Then all point-pushing maps are in $\ptI$, since they are all conjugate in $\ptM$.

\subsection{The Jacobian}\label{ss:jacobianprelim}
Of course, the study of Jacobian varieties of surfaces has a long and rich history in algebraic geometry.
However, we only consider the Jacobian from topological and real-geometric perspectives.
In keeping with these perspectives, we directly prove some preliminary results that can be quoted from sources in algebraic geometry.

\subsubsection{Conventions}
\label{sss:jacconventions}
As in the introduction, we define the Jacobian $X$ of $\Sigma$ as 
\[X=H_1(\Sigma;\R)/H_1(\Sigma;\Z).\]
Note that this is a $2g$--dimensional torus.

We use the basis $[x_1],\ldots,[x_g],[y_1],\ldots,[y_g]$ for $H_1(\Sigma)$ from Section~\ref{sss:basisreps}.
For $i=1,\ldots, g$, we have coordinate functions for $[x_i]$ and $[y_i]$ on $H_1(\Sigma;\R)$.
By differentiating these functions we get $1$--forms on $H_1(\Sigma;\R)$; these descend to $1$--forms on $X$ since they are $H_1(\Sigma)$--invariant.
Let $\tilde \alpha_i$ and $\tilde\beta_i$ denote the forms on $X$ built this way from $[x_i]$ and $[y_i]$ respectively.

Pick a set of $1$--forms $\alpha_1,\ldots,\alpha_g,\beta_1,\ldots,\beta_g$ that are basis representatives for $H^1(\Sigma;\R)$, satisfying the following conditions:
\begin{itemize}
\item $\{[\alpha_i],[\beta_i]\}_i$ is evaluation-dual to $\{[x_i],[y_i]\}_i$, meaning $\int_{x_i}\alpha_i=\int_{y_i}\beta_i=1$ and the other evaluations are $0$;
\item for each $i$, $\alpha_i\wedge\beta_i$ is a non-negative multiple of $\omega_\Sigma$;
\item for each point $p$ in $\Sigma$, there is some $\alpha_i\wedge\beta_i$ that does not vanish at $p$.
\end{itemize}

\begin{lemma}
It is possible to construct $\alpha_i$, $\beta_i$ as above.
\end{lemma}

\begin{proof}
Let $\alpha_i$ be a form representing the class in $H^1(\Sigma;\R)$ that is evaluation-dual to $[x_i]$ with respect to the given $H_1(\Sigma)$--basis, such that $\alpha_i$ has only finitely many vanishing points.
Further, we demand that these vanishing points be distinct for the different $\alpha_i$.
Temporarily pick a metric, and define $\beta_i$ to be $*(\alpha_i)$, where $*$ is the Hodge operator with respect to this metric.
Then $\{\alpha_i,\beta_i\}_i$ satisfy the conditions.
\end{proof}

\subsubsection{Properties of the Jacobian}

\begin{lemma}
The algebraic intersection form on $H_1(\Sigma;\R)$ defines a symplectic form $\omega_X$ on $X$.
We have
\[\omega_X=\sum_i\tilde\alpha_i\wedge\tilde\beta_i.\]
\end{lemma}

\begin{proof}
The tangent space to $X$ at $0$ is canonically identified with $H_1(\Sigma;\R)$, and the tangent space to $X$ at any other point can be canonically translated to $0$ by left-multiplication.
So since the algebraic intersection form is a nondegenerate alternating form on $H_1(\Sigma;\R)=T_0X$, it extends smoothly to a nondegenerate, alternating differential $2$--form on all of $X$.

The expression of $\omega_X$ in terms of $\{\tilde\alpha_i,\tilde\beta_i\}_i$ follows from evaluations.
\end{proof}

\begin{lemma}\label{le:jacaxn}
The action of $\Symps$ on $\Sigma$ induces an action on $X$ that is symplectic (for $\omega_X$) and preserves the basepoint $0$.
\end{lemma}

\begin{proof}
The action of $\Symps$ on $\Sigma$ induces an action on $H_1(\Sigma;\R)$ that preserves $H_1(\Sigma;\Z)$ and therefore descends to $X$.
Since symplectomorphisms of $\Sigma$ act on $H_1(\Sigma;\R)$ in a way that preserves intersection number, the action on $X$ preserves $\omega_X$ at $0$.
Since the action on $H_1(\Sigma;\R)$ is linear, the derivative is the same at every point and the action on $X$ is symplectic.
\end{proof}

\begin{definition}\label{de:jacaxn}
Let $\rho\co\Symps\curvearrowright (X,\omega_X)$ denote the symplectic, basepoint-preserving action induced by the action of $\Diff(\Sigma,*)$ on $H_1(\Sigma)$.
\end{definition}

\begin{definition}
The \emph{Abel--Jacobi map} $J\co\Sigma\to X$
is the map sending $p\in\Sigma$ to the image in $X$ of the point 
\[\sum_{i=1}^g\big((\int_*^p\alpha_i)[x_i]+(\int_*^p\beta_i)[y_i]\big)\]
in $H_1(\Sigma;\R)$,
where the integrals are taken over any smooth arc from $*$ to $p$.
\end{definition}
This formula certainly gives a well-defined map from the universal cover $\widetilde\Sigma$; since evaluating this sum of integrals along loops in $\Sigma$ yields elements of $H_1(\Sigma;\Z)$, it descends to a well-defined map on $\Sigma$.
It is also immediate that $J(*)=0$.
Of course, $J$ depends on the choices of $\{\alpha_i,\beta_i\}_i$ and the basepoint~$*$. 

\begin{lemma}\label{le:pullbacks}
The map $J$ is an immersion and for each $i$,
\[J^*(\tilde\alpha_i)=\alpha_i\quad\text{and}\quad J^*(\tilde\beta_i)=\beta_i.\]
\end{lemma}

\begin{proof}
It follows from elementary calculus that for any tangent vector $v$ to any point $p$ in $\Sigma$, the derivative at $p$ is
\[D_pJ(v)=((\alpha_1)_p(v),\ldots,(\alpha_g)_p(v),(\beta_1)_p(v),\ldots,(\beta_g)_p(v)).\]
The statement about pullbacks follows immediately.
So if $J$ fails to be an immersion at $p$, then $\alpha_i\wedge\beta_i$ is trivial at $p$ for every $i$, contradicting our construction of $\{\alpha_i,\beta_i\}_i$.
\end{proof}

\begin{lemma}
The form 
\[\omega'_\Sigma=\sum_{i=1}^g \alpha_i\wedge\beta_i\]
is a symplectic form with $\mathrm{Area}(\Sigma,\omega'_\Sigma)=g$.
The map $J\co (\Sigma,\omega'_\Sigma)\to(X,\omega_X)$ is a symplectic map.
\end{lemma}

\begin{proof}
Since each $\alpha_i\wedge\beta_i$ is a non-negative multiple of $\omega_\Sigma$ and the $\alpha_i\wedge\beta_i$ do not simultaneously vanish at any point, $\omega'_\Sigma$ is a multiple of $\omega_\Sigma$ by a positive function.
Therefore $\omega'_\Sigma$ is a symplectic form.
Since $\{[\alpha_i],[\beta_i]\}_i$ is evaluation-dual to $\{[x_i],[y_i]\}_i$ and $\ai([x_i,y_i])=1$, it follows from Poincar\'e duality that $\int_\Sigma\alpha_i\wedge\beta_i=1$.
Then $\mathrm{Area}(\Sigma,\omega'_\Sigma)=g$.
By Lemma~\ref{le:pullbacks}, $J$ pulls back $\omega_X$ to $\omega'_\Sigma$.
\end{proof}

We also need one further property of $J$.
\begin{lemma}\label{le:abmap}
The fundamental group $\pi_1(X,0)$ is canonically identified with $H_1(\Sigma;\Z)$ and the induced map $J_*\co \pi_1(\Sigma,*)\to H_1(\Sigma;\Z)$ is the Hurewicz map (the abelianization map of $\pi_1(\Sigma,*)$).
The induced action $\rho_*$ of $\Diff(\Sigma,*)$ on $\pi_1(X,0)$ is the usual action on $H_1(\Sigma;\Z)$.
\end{lemma}
\begin{proof}
Since $X=H_1(\Sigma;\R)/H_1(\Sigma;\Z)$ and $H_1(\Sigma;\R)$ is contractible, $\pi_1(X,0)$ is naturally isomorphic to $H_1(\Sigma;\Z)$.
For a based loop $a\co [0,1]\to\Sigma$, $J_*[a]\in H_1(\Sigma)$ is the far endpoint of the lift of $J\circ a$ to a path on $H_1(\Sigma;\R)$ starting at $0$.
Unpacking the definition of $J$, this endpoint is easily seen to be $[a]\in H_1(\Sigma)$.
The statement about actions also follows.
\end{proof}

\subsubsection{A convention on $\omega_\Sigma$}
By the Moser stability theorem (Theorem~\ref{th:moser}), there is a diffeomorphism pulling back $\omega'_\Sigma=\sum_i\alpha_i\wedge\beta_i$ to a constant multiple of $\omega_\Sigma$.
We can therefore pull back all of our constructions by this diffeomorphism and renormalize, and declare that $\omega_\Sigma=\omega'_\Sigma$.
Therefore $\omega_\Sigma=\sum_i\alpha_i\wedge\beta_i$, the map $J\co (\Sigma,\omega_\Sigma)\to(X,\omega_X)$ is symplectic, and  $\mathrm{Area}(\Sigma,\omega_\Sigma)=g$.

\subsection{The Johnson homomorphism}\label{ss:johnsonprelim}
The Johnson homomorphism is a map from the Torelli group $\ptI$ to an abelian group.
Johnson introduced this homomorphism in~\cite{joabquo}.
In his survey~\cite{josurv}, Johnson gives three equivalent definitions of the Johnson homomorphism;
the ``first definition" being the one in~\cite{joabquo}.
We will not use the ``first definition," but we mention for context that it is an algebraic definition, using the action of $\ptI$ on a nonabelian nilpotent quotient of $\pi_1(\Sigma)$.
Instead we explain the ``third definition" from~\cite{josurv} and quote some results on the Johnson homomorphism.
\subsubsection{The third definition}
This definition is well suited to the application in this paper.
This uses the Jacobian $X$ as defined in Section~\ref{ss:jacobianprelim}.
If $\phi\in\Diff(\Sigma,*)$ then the map $J\circ \phi\co \Sigma\to X$ induces the map $\rho(\phi)_*J_*\co\pi_1(\Sigma,*)\to\pi_1(X,0)$ on fundamental groups, where $\rho$ is the action from Definition~\ref{de:jacaxn}.
By Lemma~\ref{le:abmap}, $\pi_1(X,0)$ is $H_1(\Sigma)$ and the action is the usual one.
So if $\phi$ represents a class $[\phi]\in\ptI$, then the maps $J, J\circ \phi\co \Sigma\to X$ both induce the same map on fundamental groups.
Since $\Sigma$ and $X$ are aspherical, there is a homotopy $\bar K\co \Sigma\times[0,1]\to X$ from $J$ to $J\circ \phi$, through basepoint-preserving maps.
Let $M_\phi$ be the mapping cylinder of $\phi$, which is $M_\phi=\Sigma\times[0,1]/\mathord{\sim}$, where $(p,0)\sim(\phi(p),1)$.
It is easy to check that $\bar K$ defines a continuous map $K\co M_\phi\to X$.

\begin{definition}[Johnson~\cite{josurv}]
\label{de:thirddefn}
The Johnson homomorphism
$\johom\co\ptI\to H_3(X)\cong \bigwedge^3 H_1(\Sigma)$
is the map sending $[\phi]\in\ptI$ to $\johom([\phi])=K_*[M_\phi]$, the push-forward of the fundamental class of $M_\phi$, where $M_\phi$ is the mapping cylinder of $\phi$ and $K\co M_\phi\to X$ is induced by a homotopy from $J$ to $J\circ\phi$.
\end{definition}
This definition is independent of the choices made, which are the representative of $[\phi]$ and the homotopy $\bar K$.

\subsubsection{Equivariance and an evaluation}
The mapping class group $\ptM$ acts on $H_1(\Sigma)$ in the usual way, inducing the diagonal action on $\bigwedge^3 H_1(\Sigma)$.
Under the induced action by $\rho$, $H_3(X)$ and $\bigwedge^3 H_1(\Sigma)$ are isomorphic as $\ptM$--modules.

The following is Lemma~2D of Johnson~\cite{joabquo}.
\begin{theorem}[Johnson~\cite{josurv}]
The Johnson homomorphism is $\ptM$--equivariant:
for any $[\phi]\in\ptI$ and $[\psi]\in\ptM$,
\[\johom([\psi][\phi][\psi^{-1}])=[\psi]\cdot \johom([\phi]).\]
\end{theorem}

The following evaluation is a case of Lemma~4B of Johnson~\cite{joabquo}.
(In Johnson's notation, our $x_i$ is $b_i$ and our $y_i$ is $-a_i$.)
\begin{theorem}[Johnson homomorphism evaluation]\label{th:jhvalue}
The Johnson homomorphism has the following value:
\[
\johom(T_{x_g}T_{x'_g}^{-1})=\left(\sum_{i=1}^{g-1}[x_i]\wedge[y_i]\right)\wedge[x_g].\]
\end{theorem}
As mentioned above, $T_{x_g}T_{x'_g}^{-1}$ is a point-pushing map along a simple closed curve freely homotopic to $x_g$ that is based at $*$.

\subsubsection{The contraction}\label{sss:contraction}
The $\ptM$--module $\bigwedge^3 H_1(\Sigma)$ is not irreducible.
In fact, we can use the algebraic intersection form to define a $\ptM$--equivariant map.

\begin{definition}
There is a $\ptM$--module contraction
\[
\Phi\co \bigwedge^3 H_1(\Sigma)
\to H_1(\Sigma)\]
given on generators by
\[a\wedge b\wedge c \mapsto (\ai(a,b)c+\ai(b,c)a+\ai(c,a)b).
\]
\end{definition}
The value of $\Phi$ on $a\wedge b\wedge c$ is invariant under positive-sign permutations of $(a,b,c)$, so it is well defined.
Since $\ai$ is $\ptM$--invariant, the map of abelian groups given by this formula is indeed a map of $\ptM$--modules.
The next statement is used in the proof of Theorem~\ref{mt:hypvssec}.

\begin{corollary}\label{co:cjhvalue}
The contracted Johnson homomorphism has the following value:
\[\Phi(\johom(T_{x_g}T_{x'_g}^{-1}))= (g-1) [x_g].\]
Consequently, if $\phi$ is a point-pushing map along a simple closed curve $a$, we have
\[\Phi(\johom([\phi]))=(g-1)[a].\]
\end{corollary}

\begin{proof}
The first statement follows immediately from Theorem~\ref{th:jhvalue} and the definition of $\Phi$.
The second statement follows from the first statement, from the equivariance of $\Phi\circ\johom$, and from the fact that any point-pushing map on $\Sigma$ is conjugate to $T_{x_g}T_{x'_g}^{-1}$.
\end{proof}
We will need the following in the proof of Theorem~\ref{mt:jacvssec}.

\begin{lemma}\label{le:contractadjoint}
Under the evaluation pairings $\angb{,}\co H_1(\Sigma)\otimes H_1(\Sigma)^*\to \Z$ and $\angb{,}\co \bigwedge^3H_1(\Sigma)\otimes \bigwedge^3H_1(\Sigma)^*\to \Z$, the contraction $\Phi$ is adjoint to the map 
\[
\begin{split}
H_1(\Sigma)^*&\to \bigwedge^3H_1(\Sigma)^*\\
\alpha&\mapsto \ai\wedge\alpha.
\end{split}\]
Therefore for any $x\in \bigwedge^3 H_1(\Sigma)$ and $\alpha\in H_1(\Sigma)^*$, we have
\[\angb{\alpha,\Phi(x)}=\angb{\ai\wedge\alpha,x}.\]
\end{lemma}
\begin{proof}
Let $a,b,c$ be in our basis for $H_1(\Sigma)$ and let $d^*$ be in the evaluation-dual basis for $H_1(\Sigma)^*$.
Then it is a routine computation to verify the identity for $x=a\wedge b\wedge c$ and $\alpha=d^*$.
The lemma follows by linearity.
\end{proof}

\subsubsection{Extensions to crossed homomorphisms}
We need the following result for Theorem~\ref{mt:hypvssec} and Corollary~\ref{mc:eflmaps}.
\begin{proposition}\label{pr:ext}
The contracted Johnson homomorphism $\Phi\circ\johom$ extends to a crossed homomorphism $\ptM\to H_1(\Sigma;\R)$.
\end{proposition}

\begin{proof}
If $\kappa\co\ptM\to \bigwedge^3 H_1(\Sigma;\R)$ is a crossed homomorphism extending $\johom$, then $\Phi\circ \kappa$ is such a map.
Morita constructed such a $\kappa$ by algebraic methods in~\cite{mejo} (there is also a topological construction due to the author in~\cite{day}).
\end{proof}

Morita showed in~\cite{jacfam1}, Section~6, that $H^1(\ptM;H_1(\Sigma;\R))\cong\R$.
From this, Proposition~\ref{pr:ext} and the definition of $\ptI$, it follows that \emph{any} crossed homomorphism $\ptM\to H^1(\Sigma;\R)$ restricts to a constant multiple of $\Phi\circ\johom$ on $\ptI$.
Morita also gave a combinatorial construction of a nontrivial crossed homomorphism $\ptM\to H^1(\Sigma;\R)$ in~\cite{jacfam1}.
Finally, we note that Earle constructed a nontrivial crossed homomorphism $\ptM\to H^1(\Sigma;\R)$ in~\cite{earle}.

\subsection{Hyperbolic metrics}\label{ss:hyperbolicprelim}
\subsubsection{Existence of a metric}
Let $h$ be a hyperbolic metric such that the area form $dV_h$ is a constant multiple of $\omega_\Sigma$.
We demand that our basepoint $*$ does not lie on any closed geodesic.
We may do this since there are countably many closed $h$--geodesics and their union has measure zero.
\begin{lemma}
Given an area form $\omega_\Sigma$, there is a hyperbolic metric $h$ whose area form is a constant multiple of~$\omega_\Sigma$.
\end{lemma}

\begin{proof}
Since the genus of $\Sigma$ is greater than one, $\Sigma$ has a hyperbolic metric $h'$ (for a reference, see Chapter~9.2 of Ratcliffe~\cite{ratcliffe}).
By the Moser stability theorem (Theorem~\ref{th:moser}), there is a diffeomorphism pulling back $dV_{h'}$ to a constant multiple of $\omega_\Sigma$.
If we pull back $h'$ by this same diffeomorphism, we get another hyperbolic metric $h$ such that $dV_{h}$ is a constant multiple of~$\omega_\Sigma$.
\end{proof}

\subsubsection{Properties of hyperbolic metrics}
\begin{theorem}[Gauss--Bonnet]
On a surface with a hyperbolic metric, the area of a triangle with geodesic boundary and interior angles $\alpha, \beta, \gamma$ is
\[\pi-\alpha-\beta-\gamma.\]
Further, the total area of a hyperbolic surface $(\Sigma',h')$ with geodesic boundary  (possibly without boundary) is $-2\pi$ times the Euler characteristic:
\[\mathrm{Area}(\Sigma', dV_{h'})=-2\pi\chi(\Sigma').\]
\end{theorem}
Ratcliffe~\cite{ratcliffe} is a reference for the Gauss--Bonnet theorem: the first statement is Theorem~3.5.5, and the second is Theorem~9.3.1.
We will also need some well-known facts about hyperbolic geodesics.
These appear in Ratcliffe~\cite{ratcliffe} in Chapter~9.6. 
\begin{theorem}
Each homotopy class of closed curve in $\Sigma$ contains a unique closed $h$--geodesic.
\end{theorem}
\begin{theorem}
If a pair of homotopy classes of simple closed curves in $\Sigma$ have geometric intersection number $0$, then their geodesic representatives are disjoint.
If a pair of homotopy classes of simple closed curves in $\Sigma$ have geometric intersection number $1$, then their geodesic representatives have a single transverse intersection. 
\end{theorem}

\subsection{A homotopy-theoretic argument}\label{ss:homotopyprelim}
We need the following lemma to prove the well-definedness of $\Fljac{J}$.  
This is a refinement of a standard argument.

\begin{lemma}\label{le:relhtpy}
Suppose $Z$ is an aspherical space with a basepoint $*$, $Y$ is a connected CW-complex with a (different) basepoint $*$, $A\subset Y$ is a closed, connected subcomplex with $*\in A$, and $f_0,f_1\co (Y,*)\to (Z,*)$ are two continuous functions such that $f_0|_A=f_1|_A$ and $(f_0)_*=(f_1)_*\co \pi_1(Y,*)\to\pi_1(Z,*)$.
Then there is a homotopy $F\co Y\times [0,1]\to Z$ from $f_0$ to $f_1$ relative to~$A$.
\end{lemma}

\begin{proof}
Let $Y^k$ denote the $k$--skeleton of $Y$.
If we have constructed $F$ on $Y^{k-1}\cup A$ for $k>1$ and $C$ is a $k$--cell in $Y\bs A$, then $F|_{(\partial C\times [0,1])}$, $f_0|_C$ and $f_1|_C$ paste together to form a map $S^k\to Z$.
Since $Z$ is aspherical,  $F$ would then extend to $C$.
So it is enough to construct $F$ on $Y^1\cup A$.

Since $A$ is connected, we can build a maximal tree $B$ for $Y^1$ such that $B\cap A$ is connected.
Then there is a retraction homotopy $H\co (A\cup B)\times [0,1]\to A\cup B$ relative to $A$, meaning $H_0$ is the identity and $H_1(A\cup B)\subset A$.
Let $F|_{A\cup B}\co (A\cup B)\times [0,1]\to Z$ be the concatenation of $f_0\circ H$ followed by $f_1\circ H$ in reverse, which is a homotopy from $f_0|_{A\cup B}$ to $f_1|_{A\cup B}$ relative to $A$.
This concatenation is well defined since $f_0|_A=f_1|_A$.

Now suppose $\gamma\co[0,1]\to Y^1$ parametrizes an edge in $Y^1\bs(A\cup B)$.
We have already defined $F$ on $\gamma(0),\gamma(1)\in B$.
Pick paths $\delta_i$ in $A\cup B$ from $*$ to $\gamma(i)$ for $i=0,1$.
Since $f_0$ and $f_1$ induce the same map on fundamental groups, the concatenation
$(f_0)_*(\delta_1^{-1}\gamma\delta_0)^{-1}(f_1)_*(\delta_1^{-1}\gamma\delta_0)$
 is trivial in $\pi_1(Z,*)$.
Define the path $\eta_i(t)=H(\gamma(i),t)$ for $i=0,1$.
The loops $(f_1)_*\delta_i^{-1} \eta_i (f_0)_*\delta_i$ in are trivial  $\pi_1(Z,*)$ for $i=0,1$.
Then 
\[(f_1)_*\gamma \eta_1 (f_0)_*\gamma^{-1}\eta_0\]
is trivial in $\pi_1(Z,f_1(\gamma(1)))$.
This is exactly what we need to extend $F$ to a homotopy from $f_0$ to $f_1$ on the edge $\gamma$.
Repeating this for each edge, we extend $F$ to $Y^1$.
\end{proof}
\section{Constructions of maps on based symplectomorphisms}
\subsection{The area-difference map}
\label{ss:flsec}

First we give a more careful version of the definition of $\Flsec{s}$ from the introduction.
Fix a homomorphism $s\co H_1(\Sigma)\to Z_1(\Sigma\bs\{*\})$ that is a section to the projection $Z_1(\Sigma\bs\{*\})\to H_1(\Sigma\bs\{*\})\cong H_1(\Sigma)$.
\begin{definition}
For $\phi\in \Symps$ and $[c]\in H_1(\Sigma)$, pick a chain $C\in C_2(\Sigma\bs\{*\})$ bounding $s(\phi^{-1}_*[c])-\phi^{-1}_*s([c])$.
 The \emph{section-based area difference map} based on $s$,
\[\Flsec{s}\co \Symps\to H^1(\Sigma;\R)\]  
is defined by $\Flsec{s}(\phi)([c])=\int_C\omega_\Sigma$.
\end{definition}
We denote the support of a map or a chain by $\supp$.
Without demanding $*\notin\supp C$, this $\Flsec{s}$ would only be defined modulo $\mathrm{Area}(\Sigma,\omega_\Sigma)$; this is why we consider this map only on $\Symps$.
As mentioned in the introduction, this is a variation on McDuff's definition of an extended flux map in Remark~4.7 of~\cite{mcd}.
McDuff's definition also uses a section $s$ and a term like $\Flsec{s}$, but instead of using the basepoint to choose the chain $C$ with $\partial C=\phi_*s([c])-s(\phi_*[c])$, she chooses $C$ arbitrarily and uses a Chern number as a correction factor.

The goal of this subsection is to show the following.
\begin{proposition}\label{pr:flsecwelldef}
For a fixed choice of $s$, the map $\Flsec{s}$ is a well-defined crossed homomorphism extending $\Fl|_{\Sympso}$.
\end{proposition}
The proof of this proposition is broken up into several claims.
\begin{claim}
Let $\phi\in\Symps$ and let $[c]\in H_1(\Sigma)$.
Then $\Flsec{s}(\phi)([c])\in\R$ is well defined.
\end{claim}
\begin{proof}
Since $H_2(\Sigma\bs\{*\})=0$, there is a chain in $C_3(\Sigma\bs\{*\})$ bounding the difference of any two chains in $C_2(\Sigma\bs\{*\})$ that bound $s(\phi_*^{-1}[c])-\phi_*^{-1}s([c])$.
The claim then follows from Stokes's theorem and the fact that $\omega_\Sigma$ is closed.
\end{proof}

\begin{claim}
Let $\phi\in\Symps$.
Then $\Flsec{s}(\phi)\in H^1(\Sigma;\R)$.
\end{claim}
\begin{proof}
The claim follows immediately from the fact that the section $s$, the action of $\Symps$, and the boundary map $\partial$ are all linear.
\end{proof}

\begin{claim}
The map $\Flsec{s}\co \Symps\to H^1(\Sigma;\R)$ is a crossed homomorphism.
\end{claim}
\begin{proof}
Let $\phi,\psi\in\Symps$ and $[c]\in H_1(\Sigma)$.
Let $D_1, D_2\in C_2(\Sigma\bs\{*\})$ with $\partial D_1=s(\psi^{-1}_*[c])-\psi^{-1}_*s([c])$ and $\partial D_2=s(\phi^{-1}_*\psi^{-1}_*[c])-\phi^{-1}_*s(\psi^{-1}_*[c])$.
Note that $\partial(\phi^{-1}_*D_1+D_2)=s(\phi^{-1}_*\psi^{-1}_*[c])-\phi^{-1}_*\psi^{-1}_*s([c])$.
Since $\phi$ is a symplectomorphism, we have $\Flsec{s}(\psi)([c])=\int_{\phi^{-1}_*D_1}\omega_\Sigma$.
By definition, we have $\Flsec{s}(\phi)(\psi^{-1}_*[c])=\int_{D_2}\omega_\Sigma$, and $\Flsec{s}(\psi\phi)([c])=\int_{\phi^{-1}_*D_1+D_2}\omega_\Sigma$. 
Since $[c]$ was arbitrary, $\Flsec{s}(\psi\phi)=\psi\cdot\Flsec{s}(\phi)+\Flsec{s}(\psi)$.
\end{proof}

\begin{claim}
The map $\Flsec{s}$ restricted to $\Sympso$ is $\Fl|_{\Sympso}$.
\end{claim}
\begin{proof}
Let $\phi\in\Sympso$ and let $[c]\in H_1(\Sigma)$.
Pick a homotopy $\phi_t$ from the identity to $\phi$, through maps in $\Symps$.
Obtain a chain $C\in C_2(\Sigma\bs\{*\})$ by dragging $s([c])$ along $\phi_t$.
Because $\phi$ is homotopic to the identity, we have $s(\phi^{-1}_*[c])=s([c])$.
So $\partial C=\phi_*s(\phi^{-1}_*[c])-s([c])$.
Since $\phi$ is symplectic, $\int_{\phi^{-1}_*C}\omega_\Sigma=\int_C\omega_\Sigma$.
Then $\int_{\phi^{-1}_*C}\omega_\Sigma$ fits the definition of $\Flsec{s}(\phi)([c])$,
and is equal to $\int_C\omega_\Sigma$, which fits the definition of $\Fl(\phi)([c])$.
\end{proof}
The previous claim completes the proof of Proposition~\ref{pr:flsecwelldef}.

\begin{proposition}\label{pr:seccobound}
The cohomology class 
\[[\Flsec{s}]\in H^1(\Symps;H^1(\Sigma;\R))\]
 does not depend on $s$.
Consequently, $\Flsec{s}|_{\ISymps}$ does not depend on $s$.
\end{proposition}
\begin{proof}
Let $s_1,s_2\co H_1(\Sigma)\to Z_1(\Sigma\bs\{*\})$ be homomorphic sections to the projection map.
We aim to show that $\Flsec{s_1}-\Flsec{s_2}$ is in $B^1(\Symps;H^1(\Sigma;\R))$.
Since $s_1$ and $s_2$ are both sections to the projection map, we know that $s_1-s_2\co H_1(\Sigma)\to B_1(\Sigma\bs\{*\})$.
So we can choose a homomorphism $t\co H_1(\Sigma)\to C_2(\Sigma\bs\{*\})$ with $\partial\circ t = s_1-s_2$.
Let $\kappa\in H^1(\Sigma;\R)$ be given by $\kappa([c])=\int_{t([c])}\omega_\Sigma$.
Let $\phi\in\Symps$ and let $[c]\in H_1(\Sigma)$.
Let $C\in C_2(\Sigma)$ with $\partial C =s_1(\phi^{-1}_*[c])-\phi^{-1}_*s_1([c])$.
Since $\phi$ is symplectic, we have $\kappa([c])=\int_{t([c])}\omega_\Sigma=\int_{\phi^{-1}_*t([c])}\omega_\Sigma$.
Using this fact, we have
\[\big(\Flsec{s_1}-\delta\kappa\big)(\phi)([c])=\Flsec{s_1}(\phi)([c])-\kappa(\phi^{-1}_*[c])+\kappa([c])=\int_{C-t(\phi^{-1}_*[c])+\phi^{-1}_*t([c])}\omega_\Sigma.\]
Note that $\partial(C-t(\phi^{-1}_*[c])+\phi^{-1}_*t([c]))=s_2(\phi^{-1}_*[c])-\phi^{-1}_*s_2([c])$.
By the definition of $\Flsec{s_2}$, we have $\Flsec{s_2}(\phi)([c])=\big(\Flsec{s_1}-\delta\kappa\big)(\phi)([c])$.
Then since $\phi$ and $[c]$ were arbitrary, $\Flsec{s_1}-\delta\kappa=\Flsec{s_2}$.
Since $\delta\kappa|_{\ISymps}=0$, we have also proven the second statement.
\end{proof}

\subsection{The Jacobian extended flux map}
\label{ss:fljac}
In this subsection we freely use the notation of Section~\ref{ss:jacobianprelim}.
As in that section, we fix a basepoint-preserving, symplectic embedding $J\co(\Sigma,\omega_\Sigma)\to (X,\omega_X)$.
We start by precisely giving the definition of $\Fljac{J}$.
\begin{definition}
Let $\phi\in\Symps$
and let $\gamma\co S^1\to\Sigma$ be a smooth map.
Let $K\co \Sigma\times [0,1]\to X$ be a piecewise-smooth homotopy from $J\circ\phi^{-1}$ to $\rho(\phi^{-1})\circ J$ relative to $*$.
Let $T$ be a representative of the fundamental class of $S^1\times[0,1]$ relative to its boundary.
The \emph{Jacobian flux crossed homomorphism} is the crossed homomorphism 
\[\Fljac{J}\co \Symps\to H^1(\Sigma;\R)\]
defined by
\[\Fljac{J}(\phi)(\gamma_*[S^1])=\int_{K_*(\gamma\times \mathrm{id})_*T}\omega_X.\]
\end{definition}
As usual $[S^1]$ denotes the fundamental class.
Since $\gamma$ is arbitrary, this defines $\Fljac{J}(\phi)$ on any element of $H_1(\Sigma)$.
Note that we know $J\circ\phi^{-1}$ and $\rho(\phi^{-1})\circ J$ are in the same based homotopy class because they induce the same map on fundamental groups and $X$ is aspherical.
The goal of this subsection is to show the following.
\begin{proposition}\label{pr:fljacwelldef}
For a fixed $J$, $\Fljac{J}$ is a well-defined crossed homomorphism extending $\Fl|_{\Sympso}$.
\end{proposition}

The proof is broken up into three claims, which appeal to the following two lemmas.

\begin{lemma}\label{le:classchoice}
Suppose $T_1,T_2\in C_2(S^1\times [0,1])$ are two representatives of the fundamental class relative to the boundary.
If $K\co S^1\times[0,1]\to X$, then
\[\int_{K_*T_1}\omega_X=\int_{K_*T_2}\omega_X.\]
\end{lemma}
\begin{proof}
For $i=0,1$, we can find $C_i\in C_2(S^1\times\{i\})$ with $T_1-T_2+C_0+C_1\in Z_2(S^1\times[0,1])$.
Since $H_2(S^1\times[0,1])=0$, we can find a chain $D\in C_3(S^1\times[0,1])$ bounding this cycle.
Since each $C_i$ is supported on a $1$--dimensional submanifold, we know that $K^*\omega_X$ is degenerate on this submanifold and
\[0=\int_{C_i}K^*\omega_X=\int_{K_*C_i}\omega_X.\]
So by Stokes's theorem, we are done:
\[0=\int_{K_*D}d\omega_X=\int_{\partial K_*D}\omega_X=\int_{K_*T_1}\omega_X-\int_{K_*T_2}\omega_X.\qedhere\]
\end{proof}

Fix a representative $T\in Z_1(S^1)$ of the fundamental class $[S^1]$.
Let $i_0, i_1\co S^1\to S^1\times[0,1]$ be the time-zero and time-one inclusions.
Let $T'\in C_2(S^1\times [0,1])$ represent the fundamental class of $S^1\times[0,1]$ relative to its boundary, with $\partial T'=(i_1)_*T-(i_0)_*T$.
By Lemma~\ref{le:classchoice}, our constructions do not depend on the choices we just made.

\begin{lemma}\label{le:choiceofc}
Let $[c]\in H_1(\Sigma)$ and let $K\co\Sigma \times [0,1]\to X$ with symplectic endpoint maps $K_0,K_1\co\Sigma\to X$.
If $\gamma_0,\gamma_1\co S^1\to \Sigma$ are two maps with $(\gamma_0)_*[T]=(\gamma_1)_*[T]=[c]$, then
\[\int_{K_*(\gamma_0\times\mathrm{id})_*T}\omega_X=\int_{K_*(\gamma_\times\mathrm{id})_*T}\omega_X.\]
\end{lemma}

\begin{proof}
Choose a point $p\in\Sigma\bs (\cup_{i=1,2}\supp (\gamma_i)_*T)$ and a chain $C\in C_2(\Sigma\bs \{p\})$ with $\partial C=(\gamma_1)_*T-(\gamma_0)_*T$.
Let $j_0,j_1\co \Sigma\to \Sigma\times[0,1]$ be the time-zero and time-one inclusion maps.
Let \[C'=(j_0)_*C-(j_1)_*C+(\gamma_1\times \mathrm{id})_*T'-(\gamma_0\times \mathrm{id})_*T'\in C_2((\Sigma\bs \{p\})\times[0,1]).\]
Since $\partial(C')=0$ and $H_2((\Sigma\bs \{p\})\times[0,1])=0$, there is a chain $D\in C_3(\Sigma\times[0,1])$ with $\partial D=C'$.
Then
\[\partial (K_*D)=(K_0)_*C-(K_1)_*C+K_*(\gamma_1\times \mathrm{id})_*T'-K_*(\gamma_0\times \mathrm{id})_*T'.\]
Since $K_1$ and $K_0$ are symplectic, we have
\[\int_{(K_0)_*C}\omega_X=\int_C\omega_\Sigma=\int_{(K_1)_*C}\omega_X.\]
So by Stokes's theorem,
\[0=\int_{K_*D}d\omega_X=\int_{\partial K_*D}\omega_X=\int_{K_*(\gamma_1\times \mathrm{id})_*T'}\omega_X-\int_{K_*(\gamma_0\times \mathrm{id})_*T'}\omega_X.\qedhere\]
\end{proof}

\begin{claim}
For $\phi\in\Symps$ and $[c]\in H_1(\Sigma)$, the value $\Fljac{J}(\phi)([c])$ in $\R$ is well defined.
\end{claim}

\begin{proof}
According to Lemma~\ref{le:classchoice}, the choice of $T'$ does not matter.
Since $\phi^{-1}$, $\rho(\phi^{-1})$ and $J$ are all symplectic, Lemma~\ref{le:choiceofc} applies to any homotopy $K$ from $J\circ\phi^{-1}$ to $\rho(\phi^{-1})\circ J$ and the choice of $\gamma\co S^1\to\Sigma$ with $\gamma_*[T]=[c]$ also does not matter.
The only remaining choice is the choice of homotopy.

Fix a choice of $\gamma$ and let $K,K'\co \Sigma\times[0,1]\to X$ be two different piecewise-smooth homotopies relative to the basepoint from  $J\circ \phi^{-1}$ to $\rho(\phi^{-1})\circ J$.
We invoke Lemma~\ref{le:relhtpy} with $Z=X$, $Y=\Sigma\times[0,1]$, $A=(\Sigma\times\{0,1\})\cup(\{*\}\times[0,1])$, $f_0=K$ and $f_1=K'$;
this gives us a homotopy $L\co\Sigma\times[0,1]^2\to X$ from $K$ to $K'$ relative to $\Sigma\times\{0,1\}$ and $\{*\}\times[0,1]$.
By approximation theory, we may assume this $L$ is piecewise-smooth.
Since $L$ is relative to $\Sigma\times\{0,1\}$, we know $L(s,0,t)=J(\phi^{-1}(s))$ and $L(s,0,t)=\rho(\phi^{-1})(J(s))$ for any $s\in\Sigma$, $t\in[0,1]$.
Then we can find a representative $T''$ of the fundamental class of $S^1\times[0,1]^2$ relative to its boundary, such that
\begin{align*}
\partial (L_*(\gamma\times \mathrm{id})_*T'')=& K'_*(\gamma\times \mathrm{id})_*T'-K_*(\gamma\times \mathrm{id})_*T'\\
&\quad+J_*\phi^{-1}_*(\gamma\times \mathrm{id})_*T'-\rho(\phi^{-1})_*J_*(\gamma\times \mathrm{id})_*T'.
\end{align*}
The maps $J\circ\phi^{-1}\circ(\gamma\times \mathrm{id})$ and $\rho(\phi^{-1})\circ J\circ (\gamma\times \mathrm{id})$ factor through maps $S^1\to X$.
This means that 
\[\int_{J_*\phi^{-1}_*(\gamma\times \mathrm{id})_*T'}\omega_X=\int_{\rho(\phi^{-1})_*J_*(\gamma\times \mathrm{id})_*T'}\omega_X =0,\]
since these integrals can be computed on $S^1$, where the pull-back of $\omega_X$ is trivial.
So by Stokes's theorem
\[
0=\int_{L_*T''}d\omega_X=\int_{\partial L_*T''}\omega_X
=\int_{K'_*(\gamma\times \mathrm{id})_*T'}\omega_X-\int_{K_*(\gamma\times \mathrm{id})_*T'}\omega_X.
\qedhere\]
\end{proof}

\begin{claim}
The map $\Fljac{J}$ is a crossed homomorphism.
\end{claim}
\begin{proof}
Let $\phi,\psi\in\Symps$ and let $\gamma\co S^1\to \Sigma$.
Let $K_\phi, K_\psi\co \Sigma\times[0,1]\to X$, with $K_\phi$ a smooth homotopy from  $J\circ \phi^{-1}$ to $\rho(\phi^{-1})\circ J$ and
$K_\psi$ a smooth homotopy from $J\circ \psi^{-1}$ to $\rho(\psi^{-1})\circ J$.
Let $K\co \Sigma\times[0,1]\to X$ be the concatenation of $K_\psi\circ(\phi^{-1}\times \mathrm{id})$ followed by $\rho(\psi^{-1})\circ K_\phi$.
This makes sense because $K_\psi\circ(\phi^{-1}\times \mathrm{id})$ ends at $\rho(\psi^{-1})\circ J\circ \phi^{-1}$, which is where $\rho(\psi^{-1})\circ K_\phi$ begins.
Then $K$ is a piecewise-smooth homotopy from $J\circ (\phi\psi)^{-1}$ to $\rho(\phi\psi)^{-1}\circ J$.

Note that
\[\int_{(K_\psi)_*(\phi^{-1}\times \mathrm{id})_*(\gamma\times \mathrm{id})_*T'}\omega_X
=\Fljac{J}(\psi)(\phi^{-1}_*\gamma_*[T]).\]
Since $\Symps$ acts symplectically on $X$, 
\[\int_{\rho(\psi^{-1})_*(K_\phi)_*(\gamma\times \mathrm{id})_*T'}\omega_X=\int_{(K_\phi)_*(\gamma\times \mathrm{id})_*T'}\omega_X=\Fljac{J}(\phi)(\gamma_*[T]).\]
Then we have
\begin{align*}
\Fljac{J}(\phi\psi)(\gamma_*[T])&=\int_{K_*(\gamma\times \mathrm{id})_*T'}\omega_X\\
&=\int_{\rho(\psi^{-1})_*(K_\phi)_*(\gamma\times \mathrm{id})_*T'}\omega_X+\int_{(K_\psi)_*(\phi^{-1}\times \mathrm{id})_*(\gamma\times \mathrm{id})_*T'}\omega_X\\
&=\Fljac{J}(\psi)(\phi^{-1}_*\gamma_*[T])+\Fljac{J}(\phi)(\gamma_*[T]).
\end{align*}
Then $\Fljac{J}(\phi\psi)=\phi\cdot\Fljac{J}(\psi)+\Fljac{J}(\phi)$, since $\gamma$ was arbitrary.
\end{proof}

\begin{claim}
The map $\Fljac{J}$ agrees with $\Fl$ on $\Sympso$.
\end{claim}
\begin{proof}
Let $\gamma\co S^1\to \Sigma$ and let $\phi\in\Sympso$.
Then $\rho(\phi^{-1})$ is the identity.
Let $K\co \Sigma\times[0,1]\to \Sigma$ be a homotopy through $\Sympo$ from $\phi^{-1}$ to the identity.
Then $\phi\circ K$ is a homotopy through $\Sympo$ from the identity to $\phi$, so
\[\Fl(\phi)(\gamma_*[T])=\int_{\phi_*K_*(\gamma\times \mathrm{id})_*[T']}\omega_\Sigma=\int_{K_*(\gamma\times \mathrm{id})_*[T']}\omega_\Sigma.\]
However, $J\circ K$ is a homotopy from $J\circ \phi^{-1}$ to $\rho(\phi^{-1})\circ J$, and
\[\Fljac{J}(\phi)(\gamma_*[T])=\int_{J_*K_*(\gamma\times \mathrm{id})_*[T']}\omega_X=\int_{K_*(\gamma\times \mathrm{id})_*[T']}J^*\omega_X.\]
Since $J^*\omega_X=\omega_\Sigma$, this proves the claim.
\end{proof}
The previous claim finishes the proof of Proposition~\ref{pr:fljacwelldef}.

\begin{proposition}\label{pr:jaccobound}
The cohomology class 
\[[\Fljac{J}]\in H^1(\Symps;H^1(\Sigma;\R))\]
 does not depend on $J$.
Consequently, $\Fljac{J}|_{\ISymps}$ does not depend on~$J$.
\end{proposition}
\begin{proof}
Let $J,J'\co (\Sigma,*)\to (X,0)$ be two different choices of Abel--Jacobi map.
Then there is a smooth homotopy $L\co\Sigma\times[0,1]\to X$ from $J$ to $J'$ relative to $*$.
Let $\kappa\in H^1(\Sigma,\R)$ be given by
\[\kappa(\gamma_*[T])=\int_{L_*(\gamma\times \mathrm{id})_*T'}\omega_X\]
for $\gamma\co S^1\to \Sigma$.
Then by Lemma~\ref{le:classchoice}, $\kappa$ does not depend on the choice of $T'$, and since $J$ and $J'$ are symplectic, Lemma~\ref{le:choiceofc} applies and $\kappa(\gamma_*[T])$ depends only on $\gamma_*[T]$, not on $\gamma$.
Let $K, K'\co \Sigma\times [0,1]\to X$ be homotopies relative to basepoints from $J\circ\phi^{-1}$ to $\rho(\phi^{-1})\circ J'$ and from $J\circ\phi^{-1}$ to $\rho(\phi^{-1})\circ J'$, respectively.
By Lemma~\ref{le:relhtpy}, there is a smooth homotopy $\widetilde L\co\Sigma\times[0,1]^2\to X$ from $L\circ(\phi^{-1}\times \mathrm{id})$ to $\rho(\phi^{-1})\circ L$, relative to $(\Sigma\times\{0,1\})\cup(*\times[0,1])$.

Now fix $\gamma\co S^1\to \Sigma$ and $\phi\in\Symp$. 
We can find a representative $T''$ of the fundamental class of $C_3(\Sigma\times [0,1]^2)$ relative to its boundary with
\begin{align*}
\partial \widetilde L_*T''=& K_*(\gamma\times \mathrm{id})_*T'-K'_*(\gamma\times \mathrm{id})_*T'\\
&\quad+\rho(\phi^{-1})_*L_*(\gamma\times \mathrm{id})_*T'-L_*(\phi^{-1}\times \mathrm{id})_*(\gamma\times \mathrm{id})_*T'.
\end{align*}
Note that 
\[\int_{L_*(\phi^{-1}\times \mathrm{id})_*(\gamma\times \mathrm{id})_*T'}\omega_X=\kappa(\phi^{-1}_*\gamma_*[T]).\]
Since $\rho(\phi)^*\omega_X=\omega_X$, we have
\[\int_{\rho(\phi^{-1})_*L_*(\gamma\times \mathrm{id})_*T'}\omega_X=\kappa(\gamma_*[T]).\]
So by Stokes's theorem,
\[
\begin{split}
0&=\int_{\widetilde L_*D}d\omega_X\\
&=\int_{K'_*(\gamma\times \mathrm{id})_*T'}\omega_X-\int_{K_*(\gamma\times \mathrm{id})_*T'}\omega_X+\kappa(\gamma_*[T])-\kappa(\phi^{-1}_*\gamma_*[T]).
\end{split}
\]
Since $\gamma$ and $\phi$ were arbitrary, this shows that $\Fljac{J}-\Fljac{J'}=\delta\kappa$, the coboundary.
\end{proof}

\section{Differences of crossed homomorphisms related to flux}
\subsection{The first difference theorem}
\label{ss:diffmetsec}
The goal of this section is to prove Theorem~\ref{mt:hypvssec}.

\begin{lemma}\label{le:unISymp}
All extended flux maps on $\Symp$ restrict to the same map on $\ISymp$.
\end{lemma}

\begin{proof}
Kotschick--Morita proved in~\cite{km}, Theorem~2, that there is a unique cohomology class of extended flux maps in $H^1(\Symp,H^1(\Sigma;\R))$.
So any two extended flux maps differ by a coboundary.
Since $\ISymp$ is the kernel of the action $\Symp\curvearrowright H^1(\Sigma;\R)$, any such coboundary is trivial on $\ISymp$.
\end{proof}

\begin{lemma}\label{le:hamtrans}
The group $\Ham$ acts transitively on $\Sigma$ and contains a point-pushing map for each homotopy class of simple closed curve on $\Sigma$.
\end{lemma}

\begin{proof}
Let $\gamma\co [0,1]\to \Sigma$ parametrize a smooth simple closed curve $a$ based at $*$.
Let $A$ be the annulus $[-r,r]\times S^1$, where $S^1$ is $[0,\ell]/\mathord{\sim}$ for some $r,\ell>0$, with product area form $\omega_A$.
Let $N$ be a regular neighborhood of $\gamma$ with a symplectic map $(N,\omega_\Sigma)\to (A,\omega_A)$ carrying $\gamma$ to $\{0\}\times S^1$.

Let $f\co [-r,r]\to\R$ be a smooth function with the following properties:
$f(-r)=f(r)=0$;
all derivatives of $f$ are zero at $-r$ and $r$;
$f(0)=1$; and
$\int_{-r}^rf(x)dx=0$.
Such a function can easily be constructed as a sum of bump functions.
Let $H\co [0,1]\times A\to A$ be the homotopy with $H_t$ sending $(x,y)$ to $(x, y+\ell t f(x))$.
It is immediate that $H_t$ is area-preserving for each $t$.
Pull back $H_t$ to $N$ and extend by the identity to get a symplectomorphism $\phi_t$.
It follows that 
\[\Fl(\phi_t)([b])=\ai([b],[a])\int_{-r}^r \ell t f(x)dx=0,\]
for any $[b]\in H_1(\Sigma)$.
So each $\phi_t\in\Ham$.
Note that $\phi_1$ is a point-pushing map for $\gamma$.
Further note that $\phi_t(*)=\gamma(t)$ for any $t$.
Since $\gamma$ was arbitrary, we may take $\gamma$ to hit any point on $\Sigma$, so that $\Ham$ acts transitively on~$\Sigma$.
\end{proof}

\begin{lemma}\label{le:unext}
Every crossed homomorphism $F\co \Symps\to H^1(\Sigma;\R)$ that agrees with $\Fl$ on $\Symps\cap\Sympo$ extends uniquely to an extended flux map $\Symp\to H^1(\Sigma;\R)$.
\end{lemma}

\begin{proof}
For each $\phi\in \Symp$,
pick some $\psi_\phi\in \Ham$ sending $\phi(*)$ to $*$ (this is possible by Lemma~\ref{le:hamtrans}).
Then $\psi_\phi\phi\in \Symps$.
Define $\tilde F(\phi)=F(\psi_\phi\phi)$.
Note that $\tilde F(\phi)$ is well defined: if $\psi_1$ and $\psi_2$ both send $\phi(*)$ to $*$, then
\[\begin{split}
F(\psi_1\phi)-F(\psi_2\phi)&=
F(\psi_1\phi)+(\psi_2\phi)\cdot F(\phi^{-1}\psi_2^{-1})\\
&=F(\psi_1\phi)+(\psi_1\phi)\cdot F(\phi^{-1}\psi_2^{-1})\\
&=F(\psi_1\phi\phi^{-1}\psi_2^{-1})=\Fl(\psi_1\psi_2^{-1})=0.\\
\end{split}\]

Further, $\tilde F$ is a crossed homomorphism.
For $\phi_1,\phi_2\in \Symp$ and  $\psi_1,\psi_2\in \Ham$ with $\psi_i\phi_i\in \Symps$ for $i=1,2$, we have
\[\tilde F(\phi_1\phi_2)
=F((\psi_1\phi_1\psi_2\phi_1^{-1})\phi_1\phi_2)
=F(\psi_1\phi_1\psi_2\phi_2)\]
since $\psi_1\phi_1\psi_2\phi_1^{-1}\in \Ham$ sends $\phi_1\phi_2(*)$ to $*$.
So $\tilde F(\phi_1\phi_2)=\phi_1\cdot \tilde F(\phi_2)+\tilde F(\phi_1)$.

Now suppose $\tilde F'\co\Symp\to H^1(\Sigma;\R)$ is another crossed homomorphism extending $F$.
Let $\phi\in \Symp$ and let $\psi\in\Ham$ with $\psi\phi\in \Symps$.
Then
\[\begin{split}
\tilde F'(\phi)&=\tilde F'(\psi)+\psi\cdot\tilde F'(\phi)\\
&=\tilde F'(\psi\phi)=F(\psi\phi)=\tilde F(\phi).
\end{split}\]

Finally, we note that for $\phi\in\Sympo$, we have $\tilde F(\phi)=\Fl(\phi)$, so $\tilde F$ is an extended flux map.
\end{proof}

\begin{lemma}\label{le:sympssympogens}
The group $\Symps\cap\Sympo$ is generated by the union of $\Sympso$ and a finite set of point-pushing maps in $\Ham$.
\end{lemma}

\begin{proof}
The group $\Symps\cap\Sympo$ maps to $\ptM$ via the map $\Symps\to\ptM$.
The kernel of this map is $\Sympso$.
However, the composition $\Symps\cap\Sympo\to\ptM\to\clM$ is trivial.
Therefore $\Symps\cap\Sympo$ maps to the kernel of the Birman exact sequence (from Theorem~\ref{th:birmanexact}), which is the copy of $\pi_1(\Sigma,*)$ in $\ptM$ generated by the mapping classes of point-pushing maps along simple closed curves.
Since $\Symps\cap\Sympo$ contains point-pushing maps along all simple closed curves (by Lemma~\ref{le:hamtrans}), we have an exact sequence
\[1\to\Sympso\to\Symps\cap\Sympo\to\pi_1(\Sigma,*)\to 1.\]
Then $\Symps\cap\Sympo$ is generated by $\Sympso$ together with lifts of a finite generating set for $\pi_1(\Sigma,*)$, which we take to be Hamiltonian point-pushing maps by Lemma~\ref{le:hamtrans}.
\end{proof}

\begin{lemma}\label{le:pushmap}
Let $s$ be any section as in the definition of $\Flsec{s}$.
Let $\phi\in\Ham$ be a Hamiltonian point-pushing map around a simple closed curve~$a$.
Then for any $[b]\in H_1(\Sigma)$,
\[\Flsec{s}(\phi)([b])=g\cdot \ai([a],[b]).\]
\end{lemma}

\begin{proof}
We assume that $\phi$ is the map constructed in the proof of Lemma~\ref{le:hamtrans}.
The difference of that map and any other point-pushing map along $a$ is in $\Ham\cap\Sympso$ (on which $\Flsec{s}$ is trivial for any $s$), so we may assume this without loss of generality.
We pick a set of basis representatives $b_1,\ldots,b_{2g}$ for $H_1(\Sigma)$ such that $\ai([a],[b_1])=1$, $b_1$ is a simple closed curve intersecting $a$ transversely at a single point, and $b_i$ does not intersect $\supp \phi$ for $i\neq 1$.
We also demand that $*$ does not lie on any $b_i$.
Let $s$ be the section sending $[b_i]$ to $b_i$; since $\phi\in\ISymp$, $\Flsec{s}(\phi)$ does not depend on $s$.
Then it is enough to show that $\Flsec{s}(\phi)([b_1])=g$; since we easily have $\Flsec{s}(\phi)([b_i])=0$ for $i\neq 1$ the result will then follow by linearity.

As in the proof of Lemma~\ref{le:hamtrans}, the point-pushing map $\phi$ is supported on an annulus that we model as the annulus $A=[-r,r]\times([0,\ell]/\sim)$, and we model $\phi$ as $(x,y)\mapsto(x,y+\ell f(x))$, where $f\co[-r,r]\to\R$ is a smooth function satisfying certain properties.
The point $(0,0) \in A$ maps to the basepoint $*\in\Sigma$.
Since we are free to choose a different $b_1$ in the same homology class, we demand that $b_1$ intersects the support of $\phi$ on the image of the segment $t\mapsto (t,\ell/2)$ in $A$.
Let the arc $c$ be the intersection of $b_1$ with the support of $\phi$.
Recall that $\phi=\phi_1$ of a homotopy $\phi_t$ from $\mathrm{id}$ to $\phi$, supported on the same annulus for all $t$.
Let $D_0\in C_2(\Sigma)$ be the chain defined by dragging $c$ along this homotopy, with $\partial D_0=\phi_*c-c$;
 specifically, $D_0$ is the push-forward of a fundamental domain for $[-r,r]\times[0,1]$ relative to its boundary under the map $(x,t)\mapsto (x,\ell/2+t\ell f(x))$ to $A$, followed by the inclusion $A\to \Sigma$.
Since $\phi\in\Ham$, we have $\int_{D_0}\omega_\Sigma=0$.

However, one can also wrap $c$ around $A$ in the opposite direction.
Let $H\co [-r,r]\times[0,1]\to A$ send $(x,t)\mapsto (x,-\ell/2+ (1-t)\ell f(x)+\ell t)$.
Let $D_1\in C_2(\Sigma)$ be the push-forward under the composition $H\co [-r,r]\times[0,1]\to A\to \Sigma$ of a fundamental class for the domain of $H$, relative to its boundary, such that $D_1-D_0$ is a fundamental class for the image of $A$, relative to its boundary.
Then there is a chain $D_2\in C_2(\Sigma)$ such that $D_1-D_0+D_2$ is a fundamental class for $\Sigma$.
Since the basepoint $*$ is not in the image of $H$, it is not in the support of $D_1$ or $D_1+D_2$.
Since $\partial(D_1+D_2)$ is $\phi_*c-c$, the chain $D_1+D_2$ has the same area as a chain $D_3\in C_2(\Sigma\bs\{*\})$ with $\partial D_3=\phi_*b_1-b_1$.
So 
\[\Flsec{s}(\phi)([b_1])=\int_{\phi^{-1}_*D_3}\omega_\Sigma=\int_{D_1+D_2}\omega_\Sigma=\int_{D_1-D_0+D_2}\omega_\Sigma=g.\qedhere\]
\end{proof}

\begin{proof}[Proof of Theorem~\ref{mt:hypvssec}]
Let $\epsilon\co \ptM\to H^1(\Sigma;\R)$ be a crossed homomorphism extending $\Phi\circ\johom$.
By Proposition~\ref{pr:ext}, such maps exist.
Consider the crossed homomorphism
\[
F=\Flsec{s}-\frac{g}{g-1}D_\Sigma^{-1}\circ \epsilon\circ p\co\Symps\to H^1(\Sigma;\R).
\]
Let $\phi$ be a Hamiltonian point-pushing map along a based simple closed curve $a$.
By Corollary~\ref{co:cjhvalue}, we know
\[\frac{g}{g-1}D_\Sigma^{-1}\circ \Phi\circ \johom\circ p(\phi)([b])=g\cdot \ai([a],[b]),\]
and by Lemma~\ref{le:pushmap}, we know $\Flsec{s}(\phi)([b])$ also equals $g\cdot \ai([a],[b])$ for any $[b]\in H_1(\Sigma)$.
So $F$ agrees with $\Fl$ on any Hamiltonian point-pushing map $\phi$.
We proved in Section~\ref{ss:flsec} that $\Flsec{s}$ agrees with $\Fl$ on $\Sympso$.
Of course, $\johom\circ p$ is trivial on $\Sympso$ and therefore $F$ agrees with $\Fl$ on $\Sympso$.
Then by Lemma~\ref{le:sympssympogens}, $F$ agrees with $\Fl$ on a generating set for $\Sympo\cap\Symps$; since it is a crossed homomorphism, $F$ agrees with $\Fl$ on $\Symps\cap\Sympo$.
Then by Lemma~\ref{le:unext}, 
 $F$ extends to an extended flux map $\tilde F$ on $\Symp$.
By Lemma~\ref{le:unISymp}, all such maps have the same restriction to $\ISymps$, and the theorem follows from the definition of $F$.
\end{proof}

\subsection{The second difference theorem}
\label{ss:pfjacvssec}
The goal of this section is to prove Theorem~\ref{mt:jacvssec}.
As in Section~\ref{sss:basisreps}, $\{x_i,y_i\}$ are $1$-cycles representing a symplectic basis.
Let $s\co H_1(\Sigma)\to Z_1(\Sigma\bs\{*\})$ send each $[x_i]$ to $x_i$ and $[y_i]$ to $y_i$.
We also fix an Abel--Jacobi map $J$.
Since Theorem~\ref{mt:jacvssec} concerns only the restrictions of $\Flsec{s}$ and $\Flmet{J}$ to $\ISymps$, it follows from Proposition~\ref{pr:seccobound} and Proposition~\ref{pr:jaccobound} the choices of $s$ and $J$ do not matter.

In this section, we fix $\phi\in\ISymps$.
As in the definition of the Johnson homomorphism (Definition~\ref{de:thirddefn}), let $M=M_\phi=\Sigma\times[0,1]/\mathord{\sim}$, where $(p,0)\sim(\phi(p),1)$.
Also as in that definition, we choose a map $K\co M\to X$, specified by a homotopy $\bar K\co \Sigma\times[0,1]\to X$ from $J$ to $J\circ \phi$.

We construct cycles in $Z_2(M)$ that are related to the difference of $\Flsec{s}$ and $\Fljac{J}$.
Momentarily fix an index $i$.
Let $\gamma\co S^1\to \Sigma$ be a loop and let $T\in Z_1(S^1)$ be a representative of the fundamental class of $S^1$, such that $s([x_i])=\gamma_*T$.
The map $\bar K\circ (\phi^{-1}\times \mathrm{id})$ is a homotopy from $J\circ \phi^{-1}$ to $J$.
Let $T'\in C_2(S^1\times[0,1])$ be a representative of the fundamental class of $S^1\times[0,1]$, relative to its boundary, such that:
\[\partial (\bar K_*(\phi^{-1}\times \mathrm{id})_*(\gamma\times \mathrm{id})_*T')=J_*\gamma_*T-J_*\phi^{-1}_*\gamma_*T.\]
Then $\Fljac{J}(\phi)([x_i])=\int_{\bar K_*(\phi^{-1}\times \mathrm{id})_*(\gamma\times \mathrm{id})_*T'}\omega_X$.
Let $f\co S^1\times [0,1]\to M$ be the map induced by $(\phi^{-1}\circ \gamma)\times \mathrm{id}$.
Let $C\in C_2(\Sigma\bs\{*\})$ be a chain bounding $s([x_i])-\phi_*^{-1}s([x_i])$, so that $\Flsec{s}(\phi)([x_i])=\int_{C}\omega_\Sigma$.
The following cycle is important to our argument:
 \[C_i=f_*T'-(i_0)_*C\in Z_2(M).\]
Note that $C_i$ is in $Z_2(M)$ because $\partial f_*T'=(i_0)_*s([x_i])-(i_0)_*\phi_*^{-1}s([x_i])$.
Define $D_i$ the same way, but with $x_i$ replaced by $y_i$.
\begin{lemma}
For each $i$, we have
\[\int_{C_i}\omega_M=\Fljac{J}(\phi)([x_i])-\Flsec{s}(\phi)([x_i]), \quad\text{and}\]
\[\int_{D_i}\omega_M=\Fljac{J}(\phi)([y_i])-\Flsec{s}(\phi)([y_i]).\]
\end{lemma}
\begin{proof}
This is a computation:
\begin{align*}
\int_{C_i}\omega_M&=\int_{f_*T'}\omega_M-\int_{(i_0)_*C}\omega_M\\
&=\int_{\bar K_*(\phi^{-1}\times \mathrm{id})_*(\gamma\times \mathrm{id})_*T'}\omega_X-\int_{C}\omega_\Sigma=\Fljac{J}(\phi)([x_i])-\Flsec{s}(\phi)([ x_i]).
\end{align*}
The second statement is similar.
\end{proof}

We proceed to compute the Poincar\'e duals of $\{[C_i],[D_i]\}_i$.
For clarity in the computations in this section, we use $\angb{,}$ to denote the evaluation pairing between cohomology and homology; for $[\alpha]\in H^k(M;\R)$ represented by a $k$--form $\alpha$ and $[c]\in H_k(M)$ represented by a piecewise-smooth singular cycle, we have
\[\angb{[\alpha],[c]}=\int_c\alpha.\]
Let $D_M\co H^k(M;\R)\to H_{3-k}(M;\R)$ be the Poincar\'e duality isomorphism.
We denote the fundamental classes of $M$ and $\Sigma$ by $[M]$ and $[\Sigma]$, respectively.
Recall the defining property of $D_M$: for $[\alpha]\in H^k(M)$ and $[\beta]\in H^{3-k}(M)$, we have
\begin{equation}\label{eq:pdprop}\angb{[\alpha],D_M([\beta])}=\angb{[\alpha\wedge\beta],[M]}.\end{equation}

There is a product on $H_*(M)$ given by oriented transverse intersections of representative cycles (since $M$ is $3$--dimensional, every homology class has a representative that is an embedded submanifold).
It is well known (see for example Bredon~\cite{br}, p.~367) that Poincar\'e duality intertwines this product with the wedge product on cohomology:
\begin{equation}\label{eq:intertwineprods}D_M([\alpha\wedge\beta])=D_M([\alpha])\cap D_M([\beta])\end{equation}
for any $[\alpha],[\beta]\in H^*(M)$.

Recall the $1$-forms $\tilde\alpha_1,\ldots,\tilde \alpha_g,\tilde\beta_1,\ldots,\tilde\beta_g$ from Section~\ref{sss:jacconventions}.
Let $\hat \alpha_i=K^*\tilde\alpha_i$, $\hat \beta_i=K^*\tilde\beta_i$ and let $\omega_M=K^*\omega_X$.
Note that the second-coordinate map $\Sigma\times[0,1]\to[0,1]$ induces a map $M\to S^1$; let $\theta$ be a $1$--form on $M$ that is the pullback of a representative of the orientation class in $H^1(S^1)$.

As in Section~\ref{sss:jacconventions}, the $1$--forms $\{\alpha_i,\beta_i\}_i$ are the pullbacks via $J$ of the forms $\{\tilde\alpha_i,\tilde\beta_i\}_i$, and are evaluation-dual to $\{x_i,y_i\}_i$.
We have $i_0\co\Sigma\to M$ induced from the time-zero inclusion $\Sigma\to\Sigma\times[0,1]$.
Note that $K\circ i_0=J$.
Let $\hat x_j=(i_0)_*x_j$, $\hat y_j=(i_0)_*y_j$ be the cycles on $M$.
We map $[0,1]\to \Sigma\times [0,1]$ by $t\mapsto (*,t)$; since $\phi$ fixes $*$ this defines a map $S^1\to M$.
Let $z\in Z_1(M)$ be the push-forward of a representative of the fundamental class of $S^1$ along this map.

\begin{lemma}\label{le:H1basis}
The set $\{[\hat \alpha_1],\ldots,[\hat \alpha_g],[\hat \beta_1],\ldots,[\hat \beta_g],[\theta]\}$ is a minimal generating set for $H^1(M)$, 
and the set $\{[\hat x_1],\ldots,[\hat x_g],[\hat y_1],\ldots,[\hat y_g],[z]\}$ is a minimal generating set for $H_1(M)$, both of which are torsion-free.
\end{lemma}

\begin{proof}
Since $\phi$ acts trivially on $H_1(\Sigma)$, the spectral sequences for the cohomology and homology of $M$ from the fibration $\Sigma\to M\to S^1$ degenerate into K\"unneth formulas.
The lemma follows.
\end{proof}

\begin{lemma}\label{le:inttheta}
For $\alpha$ a closed $2$--form on $M$, 
we have 
\[\int_M\alpha\wedge\theta=\int_\Sigma(i_0)^*\alpha.\]
\end{lemma}
\begin{proof}
Let $i_t\co\Sigma\to M$ be the time--$t$ inclusion.
Since the $i_t$ maps are all homotopic, the integral
$\int_\Sigma (i_t)^*\alpha$
does not depend on $t$ (by Stokes's theorem).
Then we compute the integral on $\Sigma\times[0,1]$ and the result follows by Fubini's theorem.
\end{proof}

\begin{lemma}\label{le:thetaduals}
We have
\begin{gather*}
D_M([\theta]) = (i_0)_*[\Sigma],\\
D_M([\hat \alpha_i\wedge\theta]) = -[\hat y_i], \quad\text{and}\\
D_M([\hat \beta_i\wedge\theta]) = [\hat x_i].
\end{gather*}
\end{lemma}
\begin{proof}
The first statement follows immediately from Lemma~\ref{le:inttheta}.
The second and third statements follow from Lemma~\ref{le:inttheta} and Poincar\'e duality on $\Sigma$.
\end{proof}

\begin{lemma}\label{le:indepH2}
The elements 
\[[\alpha_1\wedge\theta],\ldots,[\alpha_g\wedge\theta],[\beta_1\wedge\theta],
\ldots,[\beta_g\wedge\theta]\]
are linearly independent in $H^2(M;\R)$.
\end{lemma}
\begin{proof}
These classes are linearly independent because their Poincar\'e duals are by Lemma~\ref{le:thetaduals} and Lemma~\ref{le:H1basis}.
\end{proof}

\begin{lemma}\label{le:final}
For each $i$, we have
\[\int_{C_i}\omega_M=\angb{[\omega_M\wedge\hat \beta_i],[M]},\quad\text{and} \quad \int_{D_i}\omega_M=-\angb{[\omega_M\wedge\hat \alpha_i],[M]}.\]
\end{lemma}

\begin{proof}
We show the statement for $C_i$.
We aim to show that 
\[D_M^{-1}([C_i])=[\hat \beta_i],\]
from which the proposition immediately follows.
By shifting $i_0(\Sigma)$ to intersect transversely with $C_i$, we see that $(i_0)_*[\Sigma]\cap[C_i]=[\hat x_i]$.
Then by Lemma~\ref{le:indepH2}, Lemma~\ref{le:thetaduals} and Equation~(\ref{eq:intertwineprods}),
\[[\theta]\wedge D_M^{-1}([C_i])=[\hat \beta_i\wedge\theta].\]
Together with Lemma~\ref{le:H1basis}, this implies that for some $m\in\R$, we have
\[D_M^{-1}([C_i])=m[\theta]-[\hat \beta_i].\]
It is also apparent from the definitions that $[C_i]\cap [z]=0$.
Then applying 
Equation~(\ref{eq:intertwineprods}), we see that
\begin{align*}
0&=\angb{D_M^{-1}([C_i])\wedge D_M^{-1}([z]),[M]}
=\angb{(m[\theta]-[\hat \beta_i])\wedge D_M^{-1}([z]),[M]}\\
&=m\angb{[\theta]\wedge D_M^{-1}([z]),[M]}-\angb{[\hat \beta_i]\wedge D_M^{-1}([z]),[M]}.
\end{align*}
Since $D_M([\theta])=(i_0)_*[\Sigma]$ and $[z]\cap(i_0)_*[\Sigma]=[*]$, Equation~(\ref{eq:pdprop}) 
tells us
\[\angb{[\theta]\wedge D_M^{-1}([z]),[M]}=\angb{[\theta],[z]}=\int_z\theta=1.\]
Therefore
\[m=\angb{[\hat \beta_i]\wedge D_M^{-1}([z]),[M]}=\int_z\hat\beta_i=\int_{J_*z}\tilde \beta_i=0,\]
since $J_*z$ is supported on $\{0\}\subset X$.
Since $m=0$, this proves the statement for $C_i$, and the proof for $D_i$ is similar.
\end{proof}

Recall the contraction $\Phi\co \bigwedge^3H_1(\Sigma)\to H_1(\Sigma)$ from Section~\ref{sss:contraction}.
Using canonical isomorphisms, we will regard $\Phi$ as a map $\Phi\co H_3(X)\to H_1(\Sigma)$.
\begin{lemma}\label{le:contract}
We have
\[\Phi(\johom(p(\phi)))=\sum_{j=1}^g\big(\angb{[\omega_M\wedge\hat \alpha_j],[M]}[x_j]+\angb{[\omega_M\wedge\hat \beta_j],[M]}[y_j]\big).\]
\end{lemma}
\begin{proof}
Definition~\ref{de:thirddefn} states that $\johom(p(\phi))=K_*[M]\in H_3(X)$.
Note that $[\omega_X]$ and $\ai$ define the same element of $\bigwedge^2 H_1(\Sigma)^*$.
So Lemma~\ref{le:contractadjoint} tells us
\[\angb{[\alpha],\Phi(\johom(p(\phi)))}=\angb{[\omega_X\wedge\alpha],K_*[M]}\]
for any closed $1$--form $\alpha$ on $X$.
The lemma follows when we pull these expressions back to $M$ by $K$.
\end{proof}

\begin{proof}[Proof of Theorem~\ref{mt:jacvssec}]
Apply Poincar\'e duality to Lemma~\ref{le:contract} to get
\[D_\Sigma^{-1}(\Phi(\tau(p(\phi))))=
\sum_{j=1}^g\big(\angb{[\omega_M\wedge\hat \alpha_j],[M]}[\beta_j]-\angb{[\omega_M\wedge\hat \beta_j],[M]}[\alpha_j]\big).\]
Fix an index $i$.
Then by Lemma~\ref{le:final},
\begin{align*}
D_\Sigma^{-1}(\Phi(\tau(p(\phi))))([x_i])&=-\angb{[\omega_M\wedge\hat \beta_i],[M]}\\
&=-\int_{C_i}\omega_M=\Flsec{s}(\phi)([x_i])-\Fljac{J}(\phi)([x_i]).
\end{align*}
Similarly,
\[D_\Sigma^{-1}(\Phi(\tau(p(\phi))))([y_i])
=\Flsec{s}(\phi)([y_i])-\Fljac{J}(\phi)([y_i]).\]
This proves the theorem.
\end{proof}

\subsection{Constructing extended flux maps}\label{ss:corollary}
\begin{proof}[Proof of Corollary~\ref{mc:eflmaps}]
Let $\epsilon\co\ptM\to H_1(\Sigma;\R)$ be a crossed homomorphism extending $\Phi\circ\johom$, which exists by Proposition~\ref{pr:ext}.
Of course it follows from Theorem~\ref{mt:hypvssec} that the crossed homomorphism
\[F=\Flsec{s}-\frac{g}{g-1}D_\Sigma^{-1}\circ\epsilon\circ p\]
agrees with $\Fl$ on $\Symps\cap\Sympo$ (in fact, we have already shown this in the proof of Theorem~\ref{mt:hypvssec}).
Then by Lemma~\ref{le:unext}, we have that $F$ extends uniquely to an extended flux map on $\Symp$. 
Lemma~\ref{le:unext} also applies to the crossed homomorphism
\[\Fljac{J}+\frac{1}{g-1}D_\Sigma^{-1}\circ\epsilon\circ p\]
which agrees with $\Fl$ on $\Symps\cap\Sympo$ by Theorem~\ref{mt:jacvssec}.
\end{proof}

\section{An extended flux map via hyperbolic geometry}\label{se:hypmap}

\subsection{The hyperbolic metric extended flux map}\label{ss:hypmap}
In this subsection, we define symmetric symplectic Dehn twists and show that they exist.
Then we proceed to prove Theorem~\ref{mt:symmsect}.
We finish by showing $\Flmet{h}$ is well defined.
We freely use the notation and conventions of Section~\ref{ss:hyperbolicprelim}.
We assume in this section, unless stated otherwise, that $g\geq 3$.

\begin{definition}
Define a \emph{symmetric symplectic Dehn twist} $t_a$ to be a symplectic representative of a Dehn twist about a simple closed curve $a$, supported on a regular neighborhood of $a$, with the following property:
for any simple closed curve $b$, there is a chain $C\in C_2(\Sigma)$, with  $\partial C = b+\ai([b],[a])a-(t_a)_*b$, signed area $\int_C\omega_\Sigma=0$, and $\supp C\subset \supp b\cup\supp t_a$.
\end{definition}

Recall the hyperbolic metric $h$ from Section~\ref{ss:hyperbolicprelim}.
\begin{lemma}\label{le:gottwists}
Every Dehn twist in $\clM$ has a representative that is a symmetric symplectic Dehn twist around an $h$--geodesic.
\end{lemma}

\begin{proof}
Let $a$ be a simple closed geodesic curve in $\Sigma$.
Let  $N_\epsilon(a)$ be the open $\epsilon$--neighborhood of $a$.
Pick $\epsilon>0$ small enough that the closure $\bar N_\epsilon(a)$ is a closed annulus. 
Let $\ell$ be the length of $a$ and let $\gamma\co \R/\ell\Z\to\Sigma$ be a unit-speed parametrization of $a$.
We have coordinates
\[\bar N_\epsilon(a)\to (\R/\ell \Z)\times[-\epsilon,\epsilon]\]
as follows.
A point $q$ maps to $(\gamma^{-1}(\mathrm{proj}_a(q)),r(q)d(q,a))$,
where $\mathrm{proj}_a$ is the closest point projection to $a$, $d$ is the distance, and $r(q)\in\{1,-1,0\}$ is 1 if $q$ is to the right of $a$, as viewed from above, and $-1$ if $q$ is to the left.
Let $f\co \R\to [0,1]$ be a non-decreasing smooth function such that
$f$ is locally constant outside of $[-\epsilon,\epsilon]$, we have $f(-x)=1-f(x)$ for all $x\in\R$, and $f(-\epsilon)=0$ (and therefore $f(\epsilon)=1$).
We have a diffeomorphism $\hat f$ of $(\R/\ell\Z)\times [-\epsilon,\epsilon]$ that fixes the boundary, given by $(t,x)\mapsto (t+\ell f(x),x)$.
Let $t_a\co \Sigma\to\Sigma$ be given by the action of $\hat f$ on $N_\epsilon(a)$ and by the identity on the rest of~$\Sigma$.

Then $t_a$ is clearly a diffeomorphism of $\Sigma$ and a Dehn twist around $a$. 
Let $w\co (\R/\ell \Z)\times(-\epsilon,\epsilon)\to \R$ be the function such that the pullback of the $2$-form $w dt\wedge dx$ on $(\R/\ell \Z)\times(-\epsilon,\epsilon)$ to $N_\epsilon(a)$ is $\omega_\Sigma$.
Since translation along $a$ is an isometry of $\bar N_\epsilon(a)$, the function $w(t,x)$ is constant in the first coordinate.
Then $\hat f^*(w dt\wedge dx)=w dt\wedge dx$ , and therefore $t_a$ preserves~$\omega_\Sigma$.

To show that $t_a$ is symmetric, suppose that $b$ is a simple closed curve in $\Sigma$.
In fact, in showing $t_a$ is symmetric with respect to $b$, we may replace our curve $b$ with any homologous curve supported on $\supp b\cup \supp t_a$.
To see this, suppose $b'$ is a closed curve on $\supp b\cup \supp t_a$ and $C'\in C_2(\supp b\cup\supp t_a)$ with $\partial C'=b'-b$.
If we have a chain $C\in C_2(\supp b'\cup\supp t_a)$ bounding $b'+\ai([b'],[a])a-(t_a)_*b'$ with $\int_C\omega_\Sigma=0$,
then $C-C'+(t_a)_*C'$ is a chain supported on $\supp b\cup \supp t_a$, bounding $b+\ai([b],[a])a-(t_a)_*b$, with $\int_{C-C'+(t_a)_*C'}\omega_\Sigma=0$.
So we assume that $b$ is a piecewise-smooth curve that intersects $a$ minimally, that each component of $b\cap N_\epsilon(a)$ intersects $a$, and that each component of $b\cap N_\epsilon(a)$ is a geodesic segment (we allow $b$ to intersect $\partial N_\epsilon(a)$ arbitrarily).

Let $b_1,\ldots,b_k$ be the components of $b\cap N_\epsilon(a)$.
Each $b_i$ is an open geodesic segment and we denote its closure by $c_i$.
Temporarily fix an $i$.
The curves $c_i$, $(t_a)_*c_i$ and $a$ bound two closed triangular regions $R_1$ and $R_2$.
Let $C_i$ be a chain supported on $R_1\cup R_2$, such that $\partial C_i= c_i+\ai(c_i,a)a-(t_a)_*c_i$.
Let $t_0\in\R/\ell\Z$ be such that $c_i$ intersects $a$ at the point $\gamma(t_0)$.
Then $(t_a)_*c_i$ intersects $a$ at the point $\gamma(t_0+\ell/2)$.
Note that the map $\phi$ from $(\R/\ell \Z)\times[-\epsilon,\epsilon]$ to itself sending $(t,x)$ to $(2t_0-t,-x)$ induces an isometry of $\bar N_\epsilon(a)$ (if $\bar N_\epsilon(a)$ were embedded in $\R^3$, this would be an order-two rotation around the axis through $\gamma(t_0)$ and $\gamma(t_0+\ell/2)$).
The map $\phi$ stabilizes each of $a$ and $c_i$, since these are geodesics through $\gamma(t_0)$ and the derivative $D_{(t_0,0)}\phi$ is minus the identity.
Using the symmetry of $f$, it is easy to see that $\hat f\circ \phi=\phi\circ \hat f$, so $\phi$ also stabilizes $(t_a)_*c_i$.
Then $\phi$ swaps $R_1$ and $R_2$, so they have the same area. 
The orientation of $\partial C_i$ defines an orientation of $\partial R_j$, which induces an orientation on $R_j$, for $j=1,2$.
In particular, the orientation of one of these regions is the same as $\Sigma$, and the orientation of the other is the reverse of $\Sigma$.
So $\int_{C_i}\omega_\Sigma=0$, the signed area of $R_1\cup R_2$.
Then there is a chain $C'\in C_2(\supp b)$ such that the chain
\[C=C'+\sum_{i=1}^k C_i\in C_2(\supp b\cup \supp t_a)\]
satisfies $\partial C=b+\ai([b],[a])a-(t_a)_*b$, and $\int_C\omega_\Sigma=0$.
\end{proof}

The following proposition easily implies Theorem~\ref{mt:symmsect}.
We show how the theorem follows from the proposition, and then we build up to the proof of the proposition.
\begin{proposition}\label{pr:ssdtsympo}
When $g\geq 3$, if $\phi\in\Sympo$ is a composition of symmetric symplectic Dehn twists around $h$--geodesics, then $\phi\in\Ham$.
\end{proposition}

\begin{proof}[Proof of Theorem~\ref{mt:symmsect}]
Let $\sigma\co \clM\to\Symp$ be a set-map section to the projection $\clM$, sending each mapping class to a composition of symmetric symplectic Dehn twists around $h$--geodesics (such a $\sigma$ exists by Theorem~\ref{th:Dehn}).
Let $\hat\sigma_h\co \clM\to\Symp/\Ham$ be induced by $\sigma$.
For any $\phi,\psi\in \clM$, we know that $\sigma(\phi)\sigma(\psi)\sigma(\phi\psi)^{-1}\in\Sympo$ because $\sigma$ is a section to the projection.
Then by Proposition~\ref{pr:ssdtsympo}, $\sigma(\phi)\sigma(\psi)\sigma(\phi\psi)^{-1}\in\Ham$, so $\hat\sigma_h$ is a homomorphism.

If $\sigma'$ is a second choice of section satisfying the same hypotheses, then for any $\phi\in\clM$, we have $\sigma'(\phi)\sigma(\phi)^{-1}\in\Sympo$.
Again by Proposition~\ref{pr:ssdtsympo}, it is in $\Ham$.
So $\hat\sigma_h$ does not depend on the choice of $\sigma$.
\end{proof}

Our strategy to prove Proposition~\ref{pr:ssdtsympo} is to use Gervais's presentation (Theorem~\ref{th:presentation}) to describe elements of $\Sympo$ that are compositions of symmetric symplectic Dehn twists around $h$--geodesics;
then to show that these elements are in $\Ham$, we use $\Flsec{s}$ along with some new ideas we introduce below.

Following McDuff~\cite{mcd}, we define the \emph{strange homology group} $SH_1(\Sigma)=SH_1(\Sigma,\omega_\Sigma;\Z)$ to be $Z_1(\Sigma)/\mathord{\sim}$ where $c_1\sim c_2$ if there is a chain $C\in C_2(\Sigma)$ with $\partial C=c_1-c_2$ and $\int_C\omega_\Sigma=0$.
We use an extension of this concept.
Suppose that $Y\subset\Sigma$ is a piecewise-smoothly embedded simplicial complex in $\Sigma$.
We denote by $SH_1(Y)$ the group $Z_1(Y)/\mathord{\sim}_Y$ where $c_1\sim_Y c_2$ if there is a chain $C\in C_2(Y)$ with $\partial C=c_1-c_2$ and $\int_C\omega_\Sigma=0$.
Denote the $SH_1(Y)$--class of an element $c$ of $Z_1(Y)$ by $\angb{c}$, or by $\angb{c}_Y$ if there is potential for confusion.
If $Y\subset Y'$ and $c\in Z_1(Y)$, then $\angb{c}_{Y}\subset \angb{c}_{Y'}$, but the reverse inclusion does not necessarily hold.
Note that $c\mapsto \angb{c}$ is linear.
If $\phi\in\Symp$ and $\phi$ leaves $Y$ invariant, then $\angb{c}=\angb{d}$ if and only if $\angb{\phi_*c}=\angb{\phi_*d}$.
In particular, this means that $\Symp$ acts on $SH_1(\Sigma)$, and the subgroup of $\Symp$ leaving $Y$ invariant acts on $SH_1(Y)$.

We can use this notion to restate the definition of a symmetric symplectic Dehn twist.
Suppose $t_a$ is a symplectic Dehn twist around a curve $a$.
Then $t_a$ is symmetric if it is supported on a regular neighborhood $A$ of $a$ and for every simple closed curve $b$, we have
\begin{equation}\label{eq:defprop}(t_a)_*\angb{b}=\angb{b} + \ai([b],[a])\angb{a}\in SH_1(A\cup \supp b).\end{equation}
Note the similarity between the action of symmetric symplectic Dehn twists on $SH_1$ groups and the action of Dehn twists on $H_1(\Sigma)$, as shown in Equation~(\ref{eq:twistact}).

In the following, for a subsurface $\Sigma'$ of $\Sigma$, the group $\clMn(\Sigma',\partial\Sigma')$ denotes the mapping class group of $\Sigma'$ relative to its boundary, meaning the group of diffeomorphisms of $\Sigma'$ fixing $\partial\Sigma'$ pointwise modulo equivalence by homotopy relative to $\partial\Sigma'$.
The following lemma plays a key role in the proof of Proposition~\ref{pr:ssdtsympo}.
\begin{lemma}\label{le:hamdetector}
Let $\Sigma'$ be a proper subsurface of $\Sigma$ and suppose that $\phi\in\Symp$ is supported on $\Sigma'$ and $[\phi|_{\Sigma'}]\in \clMn(\Sigma',\partial\Sigma')$ is trivial.
If $H_1(\Sigma)$ has a set of basis representatives $c_1,\ldots, c_{2g}\in Z_1(\Sigma)$ with $\phi_*\angb{c_i}=\angb{c_i}\in SH_1(\Sigma'\cup \supp c_i)$ for each $i$, then $\phi\in\Ham$.
\end{lemma}
The lemma fails if we allow $\Sigma'=\Sigma$.
If $\phi\in\Sympso$ and $c\in Z_1(\Sigma)$ such that $\Fl(\phi)([c])$ is an integer multiple of $\mathrm{Area}(\Sigma,\omega_\Sigma)$, then $\phi_*\angb{c}=\angb{c}\in SH_1(\Sigma)$.
However, $\phi$ is only in $\Ham$ if $\Fl(\phi)([c])=0$ for each~$c$.

\begin{proof}[Proof of Lemma~\ref{le:hamdetector}]
Since $[\phi|_{\Sigma'}]$ is the identity class in $\clMn(\Sigma',\partial\Sigma')$, there is a smooth homotopy from $\phi|_{\Sigma'}$ to the identity (on $\Sigma'$) relative to $\partial\Sigma'$.
Pick a point $q\in\Sigma\bs(\Sigma'\cup c_1\cup\cdots\cup c_{2g})$.
Then by the Moser stability theorem (Theorem~\ref{th:moser}), 
there is a smooth homotopy from $\phi$ to the identity on $\Sigma$, through elements of $\Symp$ that fix~$q$.

Since $\Ham$ acts transitively on $\Sigma$, we can choose $\psi\in\Ham$ with $\psi(q)=*$.
Then $\psi\phi\psi^{-1}\in\Sympso$.
The hypotheses of the lemma imply that $(\psi\phi\psi^{-1})_*\angb{\psi_*c_i}=\angb{\psi_*c_i}\in SH_1(\psi(\Sigma'\cup \supp c_i))$ for each $i$.
So for each $i$, there is a $D_i\in C_2(\psi(\Sigma'\cup \supp c_i))$ with $\partial D_i =(\psi\phi)_*c_i-\psi_*c_i$ and $\int_{D_i}\omega_\Sigma=0$.
In particular, $D_i\in C_2(\Sigma\bs\{*\})$, since $\psi(q)=*$ and $q\notin (\Sigma'\cup c_i)$.
We define $s\co H_1(\Sigma)\to Z_1(\Sigma\bs\{*\})$ by setting $s([c_i])=\psi_*c_i$ and extending linearly.
The fact that each $D_i\in C_2(\Sigma\bs\{*\})$ then implies that $\Flsec{s}(\psi\phi\psi^{-1})=0$.
Then since $\Flsec{s}$ extends $\Fl|_{\Sympso}$ and $\Fl$ is $\Symp$--equivariant, we have that $\Fl(\phi)=0$.
\end{proof}

\begin{lemma}\label{le:twistdiff}
If $t_a$ and $t_a'$ are both symmetric symplectic Dehn twists around the same geodesic $a$, then $t_a^{-1}t_a'\in \Ham$.
In particular, for any $\phi_1,\phi_2\in \Symp$, the maps $\phi_1t_a\phi_2$ and $\phi_1t_a'\phi_2$ are in the same coset of $\Ham$.
\end{lemma}

\begin{proof}
Since $t_a$ and $t_a'$ are both supported on regular neighborhoods of $a$, we have that $t_a^{-1}t_a'$ is supported on a closed regular neighborhood $\Sigma'$ of $a$.  
It is immediate that $t_a^{-1}t_a'|_{\Sigma'}$ projects to the trivial element of $\clMn(\Sigma',\partial\Sigma')$.
For any simple closed curve $b$, it follows from Equation~(\ref{eq:defprop}) that
\begin{align*}
(t_a^{-1}t_a')_*\angb{b}&=(t_a^{-1})_*(\angb{b}+\ai([b],[a])\angb{a})\\
&=\angb{b}-\ai([b],[a])\angb{a}+\ai([b],[a])\angb{a}=\angb{b}
\end{align*}
where these classes are in $SH_1(\Sigma'\cup \supp b)$.
So the first statement follows from Lemma~\ref{le:hamdetector}.
The second statement follows immediately from the fact that $\Ham$ is normal in $\Symp$.
\end{proof}

Now we analyze the lifts of relations from the mapping class group that we get by replacing Dehn twists with symmetric symplectic Dehn twists around $h$--geodesics.
The next two lemmas imply that a braid relation from Gervais's presentation lifts to an element of $\Ham$.
\begin{lemma}\label{le:intersect1}
Suppose $t_a$ is a symmetric symplectic Dehn twist around the geodesic $a$, the simple closed curve $b$ is a geodesic with $|a\cap b|=1$, and suppose $c$ is the geodesic representative of $(t_a)_*b$.
Let $\Sigma'$ be a subsurface of $\Sigma$ of genus one with a single boundary component, with $\Sigma'$ containing a neighborhood of $\supp a\cup \supp b\cup \supp c$.
Then $\angb{b}+\ai([b],[a])\angb{a}=\angb{c}\in SH_1(\Sigma')$.
\end{lemma} 

\begin{proof}
We proceed by assuming that $\ai([b],[a])=1$.
Then $\ai([c],[a])=1$ and $\ai([b],[c])=1$.
Each pair of these geodesics intersects at a single transverse intersection point.
There are then two possibilities: the geodesics intersect at a single triple-intersection point, or there are three transverse double-intersections.

If we have a triple-intersection, then the geodesics cut the surface into two components: a triangle with all three vertices at the intersection point, and its complement (see Figure~\ref{fi:braidrelation}).
The angles around the intersection point include each of the interior angles of the triangle twice and no other angles.
Then the sum the interior angles of this triangle is $\pi$, and by the Gauss--Bonnet theorem, the area of this triangle is zero.
This is a contradiction.

\begin{figure}[ht!]
\labellist
\small\hair 2pt
\pinlabel $a$ [b] at 124 81 
\pinlabel $b$ [r] at 44 77 
\pinlabel $c$ [r] at 124 59
\endlabellist
\centering
\includegraphics[scale=0.65]{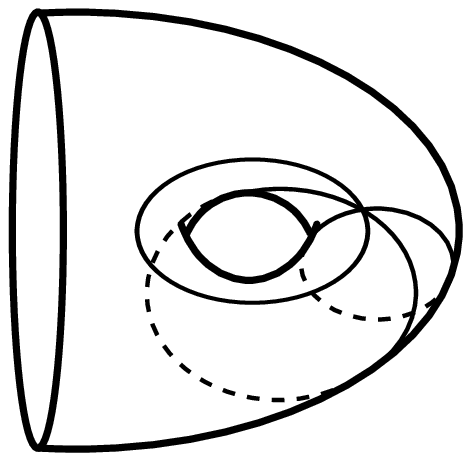}
\quad
\labellist
\small\hair 2pt
\pinlabel $R_1$ [b] at 47 106
\pinlabel $R_2$ [bl] at 125 114
\pinlabel $a$ [t] at 122 77
\pinlabel $b$ [r] at 44 77 
\pinlabel $c$ [b] at 98 103
\endlabellist
\centering
\includegraphics[scale=0.65]{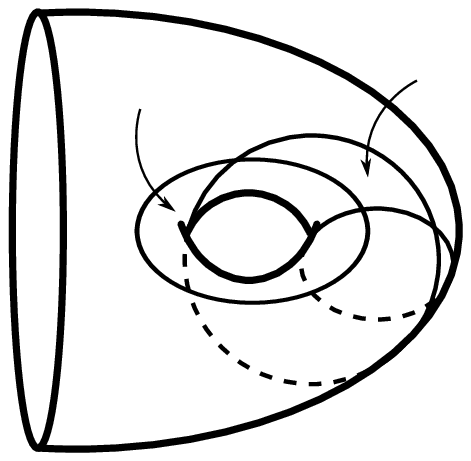}
\caption{An impossible triangle configuration of geodesics on a hyperbolic surface (\emph{left}), and a possible configuration (\emph{right}).}
\label{fi:braidrelation}
\end{figure}

So we have three double-intersections.
Then the geodesics cut the surface into three components: two closed triangles $R_1$ and $R_2$, and the rest of $\Sigma$.
As seen in Figure~\ref{fi:braidrelation}, each angle in $R_1$ is opposite to one of the angles in $R_2$.
Then $R_1$ and $R_2$ have the same angles, so by the Gauss--Bonnet theorem, they have the same area.
By the same reasoning as in the proof of Lemma~\ref{le:gottwists}, there is a chain $C$ supported on $R_1\cup R_2$ with $\partial C=b +\ai([b],[a])a-c$ and $\int_C\omega_\Sigma=0$.
A parallel argument applies if $\ai([b],[a])=-1$.
\end{proof}

\begin{lemma}\label{le:braidrelation}
Let $t_a$ be a symmetric symplectic Dehn twist around the geodesic $a$.
Let the simple closed curve $b$ be a geodesic with $|a\cap b|=1$, and let $c$ be the geodesic representative of $(t_a)_*b$.
Let $t_b$ and $t_c$ be symmetric symplectic Dehn twists around $b$ and $c$ respectively.
Then $\phi=t_c^{-1}t_at_bt_a^{-1}$ is in $\Ham$.
\end{lemma}

\begin{proof}
Again, we start by assuming that $\ai([b],[a])=1$, so that $\ai([c],[a])=1$ and $\ai([b],[c])=1$.
Let $\Sigma'$ be a subsurface of $\Sigma$ of genus one with a single boundary component, such that $\Sigma'$ contains a neighborhood of $\supp a\cup \supp b\cup \supp c$.
Then $\Sigma'$ is a proper subsurface of $\Sigma$.
By Lemma~\ref{le:twistdiff}, we may shrink the supports of our twists so that each of $t_a$, $t_b$, and $t_c$ is supported on $\Sigma'$.
Since $\phi$ is the lift of a braid relation (which holds in all mapping class groups), we know $\phi|_\Sigma'$ maps to the trivial element in $\clMn(\Sigma',\partial\Sigma')$.

We can find a set of representatives of a basis of $H_1(\Sigma)$ that consists of $a$, $b$, and $2g-2$ curves that are disjoint from $\Sigma'$.
If the cycle $x$ is disjoint from the support of $\phi$, then $\phi_*\angb{x}=\angb{x}\in SH_1(\Sigma'\cup \supp x)$.
By Equation~(\ref{eq:defprop}),
\begin{align*}
\phi_*\angb{a}&=(t_c^{-1}t_at_bt_a^{-1})_*\angb{a}=(t_c^{-1}t_at_b)_*\angb{a}\\
&=(t_c^{-1}t_a)_*(\angb{a}-\angb{b})=(t_c^{-1})_*(-\angb{b})=\angb{c}-\angb{b}
\end{align*}
where these classes are understood to be in $SH_1(\Sigma')$ (which contains the curves and the supports of the twists).
But by Lemma~\ref{le:intersect1}, we know that $\angb{c}=\angb{a}+\angb{b}\in SH_1(\Sigma')$.
So we have $\phi_*\angb{a}=\angb{a}\in SH_1(\Sigma')$.
Similarly by Equation~(\ref{eq:defprop}) and Lemma~\ref{le:intersect1},
\[
\phi_*\angb{b}=(t_c^{-1}t_at_bt_a^{-1})_*\angb{b}
=2\angb{b}+\angb{a}-\angb{c}=\angb{b}\in SH_1(\Sigma').
\]
This completes the proof if $\ai([b],[a])=1$, and a parallel argument applies if $\ai([b],[a])=-1$.
\end{proof}

\begin{figure}[ht!]
\labellist
\small\hair 1pt
\pinlabel $a_1$ [b] at 163 154
\pinlabel $a_2$ [bl] at 93 148
\pinlabel $a_3$ [r] at 111 80
\pinlabel $b$ [b] at 122 149
\pinlabel $x$ [l] at 92 174
\pinlabel $y$ [b] at 88 63
\pinlabel $d_1$ [tr] at 69 86
\pinlabel $d_2$ [tl] at 183 93
\pinlabel $d_3$ [b] at 121 187
\endlabellist
\centering
\includegraphics[scale=0.65]{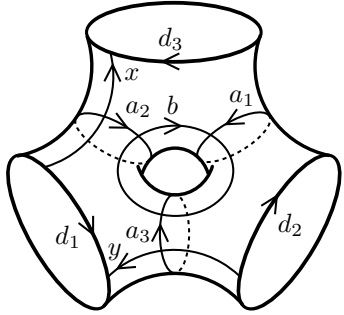}
\caption{The curves of a star relation and two reference curves.}
\label{fi:starrelation}
\end{figure}

Next we prove that a star or chain relation made out of symmetric symplectic Dehn twists around $h$--geodesics is also Hamiltonian.
\begin{lemma}\label{le:threepower}
Suppose $\Sigma'$ is a subsurface of $\Sigma$ of genus one  with three boundary components, and $a_1, a_2, a_3$ and $b$ are geodesics in $\Sigma'$ in the star relation configuration, as in Figure~\ref{fi:starrelation}.
Suppose $t_{a_1},t_{a_2},t_{a_3}$ and $t_b$ are symmetric symplectic Dehn twists around their respective curves.
Suppose we have an additional simple closed curve $x$ in $\Sigma$ and some $j$ with $|x\cap a_j|=1$ and $\ai([a_j],[x])=-1$, and suppose that $x$ does not intersect the other $a_k$ or $b$.
Let $\psi=(t_{a_1}t_{a_2}t_{a_3}t_b)^3$.
Then $\psi_*\angb{a_1}=\angb{a_1}\in SH_1(\Sigma')$, $\psi_*\angb{b}=\angb{b}\in SH_1(\Sigma')$, and $\psi_*\angb{x}=\angb{x}+\angb{a_1}+\angb{a_2}+\angb{a_3}-3\angb{a_j}\in SH_1(\Sigma'\cup \supp x)$.
\end{lemma}

\begin{proof}
This is a standard computation using Equation~(\ref{eq:defprop}), similar to the computations in Lemma~\ref{le:braidrelation} (but without any use of Lemma~\ref{le:intersect1}).
We leave this as an exercise for the reader.
\end{proof}

\begin{lemma}\label{le:starrelation}
Suppose $t_{a_1},t_{a_2},t_{a_3}$ and $t_b$ are as in Lemma~\ref{le:threepower}, the non-separating geodesics $d_1$, $d_2$ and $d_3$ are as in Figure~\ref{fi:starrelation}, and $t_{d_1}$, $t_{d_2}$ and $t_{d_3}$ are symmetric symplectic Dehn twists around the respective curves.
Let $\phi = t_{d_1}^{-1}t_{d_2}^{-1}t_{d_3}^{-1}(t_{a_1}t_{a_2}t_{a_3}t_b)^3$.
Then $\phi\in \Ham$.
\end{lemma}

\begin{proof}
Let $\Sigma'$ be a subsurface of $\Sigma$ of genus one with three boundary components, such that $\Sigma'$ contains a regular neighborhood of the union of our seven curves.
By Lemma~\ref{le:twistdiff}, we may assume that $\phi$ is supported on $\Sigma'$.
Because the $d_i$ are non-separating, there is a set of representatives for a basis for $H_1(\Sigma)$ consisting of the following:
\begin{itemize}
\item the curves $a_1$ and $b$; 
\item a simple closed curve $x$ that does not intersect $a_1$, $a_3$, or $b$, and that intersects each of $a_2$, $d_1$, and $d_3$ only once, with $\ai([x],[a_2])=-1$, $\ai([x],[d_1])=-1$ and $\ai([x],[d_3])=1$;
\item a simple closed curve $y$ that does not intersect $a_1$, $a_2$, or $b$, and that intersects each of $a_3$, $d_1$ and $d_2$ only once, with $\ai([y],[a_3])=-1$, $\ai([y],[d_1])=1$ and $\ai([y],[d_2])=-1$;
\item and $2g-4$ cycles with support disjoint from $\Sigma'$.
\end{itemize}
The curves $a_1$, $b$, $x$ and $y$ are illustrated in Figure~\ref{fi:starrelation}.
Then by Lemma~\ref{le:threepower} and Equation~(\ref{eq:defprop}), we have $\phi_*\angb{a_1}=\angb{a_1}\in SH_1(\Sigma')$, $\phi_*\angb{b}=\angb{b}\in SH_1(\Sigma')$,
\[\phi_*\angb{x}=\angb{x}+\angb{a_1}-2\angb{a_2}+\angb{a_3}+\angb{d_1}-\angb{d_3}\in SH_1(\Sigma'\cup \supp x)\]
and
\[\phi_*\angb{y}=\angb{y}+\angb{a_1}+\angb{a_2}-2\angb{a_3}-\angb{d_1}+\angb{d_2}\in SH_1(\Sigma'\cup \supp y).\]

Let $D_1$, $D_2$ and $D_3\in C_2(\Sigma')$ with $\partial D_1= a_2-a_1+d_3$,
 $\partial D_2= a_3-a_2+d_1$, and $\partial D_3=a_1-a_3+d_2$, such that each $D_i$ is supported on a subsurface of $\Sigma'$ of genus zero.
Each $D_i$ has geodesic boundary with three boundary components, and the orientation on the boundary of $D_i$ induces an orientation on this subsurface that agrees with the orientation of $\Sigma$.
So by the Gauss--Bonnet theorem, we know that $\int_{D_1}\omega_\Sigma=\int_{D_2}\omega_\Sigma=\int_{D_3}\omega_\Sigma$.
In particular, 
$\phi_*\angb{x}=\angb{x}+\angb{\partial(D_2-D_1)}$, and $\int_{D_2-D_1}\omega_\Sigma=0$, so $\phi_*\angb{x}=\angb{x}\in SH_1(\Sigma'\cup \supp x)$.
Similarly, 
$\phi_*\angb{y}=\angb{y}+\angb{\partial(D_3-D_2)}$, so $\phi_*\angb{y}=\angb{y}\in SH_1(\Sigma'\cup \supp y)$.
Then $\phi$ fixes the class in $SH_1(\Sigma'\cup c)$ of each element $c$ of our set of basis representatives for $H_1(\Sigma)$.
Since $\phi$ is a lift of the star relation, we know $1=[\phi|_{\Sigma'}]\in\clMn(\Sigma',\partial\Sigma')$ (and $\Sigma'$ is a proper subsurface of $\Sigma$), so it then follows from Lemma~\ref{le:hamdetector} that $\phi\in\Ham$.
\end{proof}

\begin{lemma}\label{le:chainrelation}
Suppose $t_{a_1},t_{a_2}$, $t_b$, $t_{d_1}$ and $t_{d_3}$ are as in Lemma~\ref{le:starrelation}, the non-separating geodesics $d_1$, and $d_3$ are as in Figure~\ref{fi:starrelation}, and we have a curve $d_2$ as in Figure~\ref{fi:starrelation} that bounds a disk.
Let $\phi = t_{d_1}^{-1}t_{d_3}^{-1}(t_{a_1}t_{a_2}t_{a_1}t_b)^3$.
Then $\phi\in \Ham$.
\end{lemma}

\begin{proof}
Let $\Sigma'$ be a  subsurface of $\Sigma$ of genus one with two boundary components, such that $\Sigma'$ contains a regular neighborhood of the union of these five curves.
By Lemma~\ref{le:twistdiff}, we assume that $\phi$ is supported on $\Sigma'$.
Let $x$ be as in the proof of Lemma~\ref{le:starrelation}.
Since $d_1$ and $d_3$ are non-separating, we have a set of basis representatives consisting of $a_1$, $b$, $x$ and $2g-3$ curves that are disjoint from the support of $\phi$.
By the same reasoning as in the proof of Lemma~\ref{le:starrelation}, we have $\phi_*\angb{a_1}=\angb{a_1}\in SH_1(\Sigma')$, $\phi_*\angb{b}=\angb{b}\in SH_1(\Sigma')$, and $\phi_*\angb{x}=\angb{x}\in SH_1(\Sigma'\cup \supp x)$ (Lemma~\ref{le:threepower} still applies, with $a_1=a_3$ since $d_2$ bounds a disk).
Then since $\phi_*$ is a lift of the chain relation, $1=[\phi|_{\Sigma'}]\in\clMn(\Sigma',\partial\Sigma')$.
Since the genus of $\Sigma$ is greater than two, we know $\Sigma'$ is a proper subsurface of $\Sigma$, and it follows from Lemma~\ref{le:hamdetector} that $\phi\in\Ham$.
\end{proof}

\begin{remark}
Lemma~\ref{le:chainrelation} and Proposition~\ref{pr:ssdtsympo} fail if $g=2$.
Suppose we have constructed $\Sigma$ from the surface in Figure~\ref{fi:starrelation} by gluing a disk into $d_2$ and gluing $d_1$ to $d_3$ in orientation-preserving fashion.
Suppose we have done this such that the reference arc $x$ becomes a simple closed curve.
Let $\phi$ be as in Lemma~\ref{le:chainrelation}.
Let $Y$ be a union of small regular neighborhoods of the marked curves.
Lemma~\ref{le:threepower} shows $\phi_*\angb{x}=\angb{x}+2\angb{a_1}-2\angb{a_2}+2\angb{d_1}$ in $SH_1(Y)$.
This means that there is an area-zero $2$--chain $D_1$, supported on $Y$, such that  $\partial D_1=x-2a_1-2a_2+2d_1-\phi_*x$.
Let $D_2$ be a fundamental class, relative to the boundary, for the subsurface bounded by the geodesics $a_1$, $a_2$, and $d_1$, with $\partial D_2=a_1+a_2-d_1$.
Pick a section $s$ with $s([x])=x$, and pick a basepoint not in $Y\cup \supp D_2$.
Then $\Fl(\phi)([x])=\Flsec{s}(\phi)([x])=\int_{D_1+2D_2}\omega_\Sigma$, which is the area of $\Sigma$ and is nonzero.
So if $g=2$, we have a product of symmetric symplectic Dehn twists around $h$--geodesics that is in $\Sympo$ but not in $\Ham$.
\end{remark}

\begin{proof}[Proof of Proposition~\ref{pr:ssdtsympo}]
Pick a set $S\subset\Symp$ of symmetric symplectic Dehn twists around simply closed geodesics, one for each free homotopy class.
By replacing each twist in the composition of $\phi$ with an element of $S\cup S^{-1}$, we get a new map $\phi'$.
By Lemma~\ref{le:twistdiff}, $\phi'$ is in the same coset of $\Ham$ as $\phi$.
The composition of $\phi'$ describes a word $\tilde w$ in $S$.
Let $w$ be the same word with each letter in $S$ replaced by its image in $\clM$.
Since $\phi'\in\Sympo$, the word $w$ represents the trivial element of $\clM$, and is a product of conjugates of the relations in Theorem~\ref{th:presentation}.
Then since $\Ham$ is normal in $\Symp$, $\phi'$ is in $\Ham$ if all of the relations in Theorem~\ref{th:presentation} are in $\Ham$, when lifted to $\Symp$ using elements of $S$.

If $a$ and $b$ are disjoint geodesics and $t_a, t_b\in S$ are twists around them, then by Lemma~\ref{le:twistdiff}, we may assume that their supports are disjoint, in which case they commute.
This means that relation~(\ref{it:comm}) is satisfied.
The other relations are satisfied because of Lemma~\ref{le:braidrelation}, Lemma~\ref{le:starrelation} and Lemma~\ref{le:chainrelation}.
\end{proof}

Recall that $\Flmet{h}(\phi)=\Fl(\phi\sigma([\phi])^{-1})$.
\begin{proposition}
The map $\Flmet{h}$ is a well-defined crossed homomorphism extending $\Fl$.
\end{proposition}
\begin{proof}
Recall that $\Fl$ is a $\Symp$--equivariant homomorphism.
Then for $\phi,\psi\in\Symp$, we have
\begin{align*}\Flmet{h}(\phi\psi)&=\Fl(\phi\psi\sigma([\phi\psi])^{-1})=\Fl(\phi\psi\sigma([\psi])^{-1}\sigma([\phi])^{-1})\\
&=\Fl(\phi\psi\sigma([\psi])^{-1}\phi^{-1}\phi\sigma([\phi])^{-1})\\
&=\phi\cdot\Fl(\psi\sigma([\psi])^{-1})+\Fl(\phi\sigma([\phi])^{-1})\\
&=\phi\cdot\Flmet{h}(\psi)+\Flmet{h}(\phi).
\end{align*}
Here we are using that $\Fl(\sigma([\phi\psi])\sigma([\psi])^{-1}\sigma([\phi])^{-1})=0$, which is true by Theorem~\ref{mt:symmsect}.
If $\phi\in\Sympo$, then $\sigma([\phi])=\sigma(1)\in\Ham$, again by Theorem~\ref{mt:symmsect}.
So $\Flmet{h}(\phi)=\Fl(\phi)$ for $\phi\in\Sympo$.
Since $h$--symmetric sections define a unique homomorphism $\hat\sigma$, it follows that if we define $\Flmet{h}$ using a different $h$--symmetric section $\sigma$, we get the same map $\Flmet{h}$.
\end{proof}

\bibliographystyle{amsplain}
\bibliography{efjh}

\noindent
Department of Mathematics 253-37\\
California Institute of Technology\\
Pasadena, CA 91125\\
E-mail: {\tt mattday@caltech.edu}
\medskip

\end{document}